\documentclass[reqno]{amsart}
\usepackage{amssymb}
\usepackage{graphicx}

\usepackage[usenames, dvipsnames]{color}
\usepackage{verbatim}
\usepackage{mathrsfs}
\usepackage{bm}
\usepackage{cite}

\numberwithin{equation}{section}

\newtheorem{theorem}{Theorem}[section]
\newtheorem{corollary}[theorem]{Corollary}
\newtheorem{lemma}[theorem]{Lemma}

\theoremstyle{definition}
\newtheorem{remark}[theorem]{Remark}

\theoremstyle{definition}
\newtheorem{definition}[theorem]{Definition}

\theoremstyle{definition}

\makeatletter
\def\dashint{\operatorname%
{\,\,\text{\bf-}\kern-.98em\DOTSI\intop\ilimits@\!\!}}
\makeatother

\def\\det{\text{det}}

\def\Xint#1{\mathchoice
 {\XXint\displaystyle\textstyle{#1}}%
 {\XXint\textstyle\scriptstyle{#1}}%
 {\XXint\scriptstyle\scriptscriptstyle{#1}}%
 {\XXint\scriptscriptstyle\scriptscriptstyle{#1}}%
 \!\int}
\def\XXint#1#2#3{{\setbox0=\hbox{$#1{#2#3}{\int}$}
  \vcenter{\hbox{$#2#3$}}\kern-.5\wd0}}

\def\dashint{\Xint-}

\def\.5{\frac{1}{2}}

\newcommand{\RN}[1]{%
  \textup{\uppercase\expandafter{\romannumeral#1}}%
}

\renewcommand{\epsilon}{\varepsilon}

\newcounter{marnote}

\begin{document}
\title[Asymptotic stability of homogeneous solutions]{Asymptotic stability of homogeneous solutions to Navier-Stokes equations under $L^{p}$-perturbations}

\author[Z. Zhao]{Zhiwen Zhao}

\address[Z. Zhao]{Beijing Computational Science Research Center, Beijing 100193, China.}
\email{zwzhao365@163.com}

\author[X. Zheng]{Xiaoxin Zheng}
\address[X. Zheng]{School of Mathematical Sciences, Beihang University, Beijing 100191, China}
\email{xiaoxinzheng@buaa.edu.cn}

%\thanks{}
%\footnote{ }

\date{\today} % delete this line to display the current date

%%% BEGIN DOCUMENT

%\tableofcontents

\begin{abstract}
It is known that there has been classified for all $(-1)$-homogeneous axisymmetric no-swirl solutions of the three-dimensional Navier-Stokes equations with a possible singular ray. The main purpose of this paper is to show that the least singular solutions among such solutions other than Landau solutions to the Navier-Stokes equations are asymptotically stable under $L^{3}$-perturbations. Moreover, we establish the $L^{q}$ decay estimate with an explicit decay rate and a sharp constant for any $q>3$. For that purpose, we first study the global well-posedness of solutions to the perturbed equations under small initial data in $L_{\sigma}^{3}$ space and the local well-posedness with any initial data in $L_{\sigma}^{p}$ spaces for $p\geq3$.

\end{abstract}

\maketitle
%\date{}
%\maketitle
%{\bf Abstract}

%{Keywords}

%{\bf Mathematics Subject Classification(2010)} $35{\rm B}33 \cdot   35{\rm B}40 \cdot  35{\rm B}44$

\section{Introduction and main results}

The motion of a steady-state incompressible viscous fluid in $\mathbb{R}^{3}$ can be modeled by the stationary Navier-Stokes equations as follows:
\begin{align}\label{NS001}
\begin{cases}
-\Delta u+u\cdot\nabla u+\nabla p=0,\\
\mathrm{div}u=0,
\end{cases}
\end{align}
where $u$ is a vector field denoting velocity and $p$ represents the pressure. Since the equations are invariant after the scaling $u(x)\rightarrow \lambda u(\lambda x)$ and $p(x)\rightarrow \lambda^{2}p(\lambda x)$ with $\lambda>0$, it is natural to find solutions which keep invariant under this scaling. These solutions are called $(-1)$-homogeneous solutions on the basis of the homogeneity of $u$, although $p$ is $(-2)$-homogeneous.

By using spherical polar coordinates $(r,\theta,\phi)$, we represent a vector field $u$ in the following form
\begin{align*}
u=u_{r}e_{r}+u_{\theta}e_{\theta}+u_{\phi}e_{\phi},
\end{align*}
where
\begin{gather*}
e_{r}=\begin{pmatrix}
\sin\theta\cos\phi\\
\sin\theta\sin\phi\\
\cos\theta
\end{pmatrix},\quad
e_{\theta}=
\begin{pmatrix}
\cos\theta\cos\phi\\
\cos\theta\sin\phi\\
-\sin\theta
\end{pmatrix},\quad
e_{\phi}=
\begin{pmatrix}
-\sin\phi\\
\cos\phi\\
0
\end{pmatrix}.
\end{gather*}
Here $r$ denotes the radial distance from the origin, $\theta$ represents the polar angle between the positive $x_{3}$-axis and the radius vector, and $\phi$ is the azimuthal angle about $x_{3}$-axis. A vector field $u$ is called axisymmetric if $u_{r},u_{\theta}$ and $u_{\phi}$ depend only on $r$ and $\theta$, and is called $\textit{no-swirl}$ if $u_{\phi}=0$.

By introducing the new functions and variables: $U_{r}:=u_{r}r\sin\theta$, $U_{\theta}:=u_{\theta}r\sin\theta$, $U_{\phi}:=u_{\phi}r\sin\theta$ and $y:=\cos\theta$, the $(-1)$-homogeneous axisymmetric no-swir solutions of \eqref{NS001} can be reduced to the following ordinary differential equation (see \cite{S1972,LLY201801})
\begin{align}\label{NS002}
(1-y^{2})U_{\theta}'+2yU_{\theta}+\frac{1}{2}U^{2}_{\theta}=c_{1}(1-y)+c_{2}(1+y)+c_{3}(1-y^{2}),\;\,\mathrm{in}\;(-1,1),
\end{align}
where $c:=(c_{1},c_{2},c_{3})\in J:=\{c\in\mathbb{R}^{3}|\,c_{1}\geq-1,c_{2}\geq-1,c_{3}\geq\bar{c}_{3}(c_{1},c_{2})\}$ with $\bar{c}_{3}(c_{1},c_{2}):=-\frac{1}{2}(\sqrt{1+c_{1}}+\sqrt{1+c_{2}})(\sqrt{1+c_{1}}+\sqrt{1+c_{2}}+2)$. It has been shown in \cite{LLY201802} that there exist $\gamma^{\pm}\in C^{0}(J;\mathbb{R})$ such that for any $(c,\gamma)\in J\times[\gamma^{-}(c),\gamma^{+}(c)]$, ODE \eqref{NS002} possesses a unique solution $U_{\theta}^{c,\gamma}$ in $C^{\infty}(-1,1)\cap C^{0}[-1,1]$ with $U_{\theta}^{c,\gamma}(0)=\gamma$, where $\gamma^{+}(c)$ and $\gamma^{-}(c)$ satisfy that $\gamma^{+}(c)>\gamma^{-}(c)$ if $c_{3}>\bar{c}_{3}(c_{1},c_{2})$, and $\gamma^{+}(c)=\gamma^{-}(c)$ if $c_{3}=\bar{c}_{3}(c_{1},c_{2})$, the values of $U_{\theta}^{c,\gamma}(\pm1)$ are given by
\begin{align*}
U_{\theta}^{c,\gamma}(-1)=&
\begin{cases}
2+2\sqrt{1+c_{1}},&\mathrm{if}\;\gamma=\gamma^{+}(c),\\
2-2\sqrt{1+c_{1}},&\mathrm{otherwise},
\end{cases}\\
U_{\theta}^{c,\gamma}(1)=&
\begin{cases}
-2-2\sqrt{1+c_{2}},&\mathrm{if}\;\gamma=\gamma^{-}(c),\\
-2+2\sqrt{1+c_{2}},&\mathrm{otherwise}.
\end{cases}
\end{align*}
Write
\begin{align*}
u^{c,\gamma}=&u_{r}^{c,\gamma}e_{r}+u_{\theta}^{c,\gamma}e_{\theta}=\frac{(U_{\theta}^{c,\gamma})'}{r}e_{r}+\frac{U_{\theta}^{c,\gamma}}{r\sin\theta}e_{\theta},\\
p^{c,\gamma}=&u_{r}^{c,\gamma}-\frac{1}{2}(u_{\theta}^{c,\gamma})^{2}=\frac{1}{r}\left((U_{\theta}^{c,\gamma})'-\frac{(U_{\theta}^{c,\gamma})^{2}}{2r\sin^{2}\theta}\right).
\end{align*}
Then $\{(u^{c,\gamma},p^{c,\gamma})|\,(c,\gamma)\in J\times[\gamma^{-}(c),\gamma^{+}(c)]\}$ compose all $(-1)$-homogeneous axisymmetric no-swirl solutions for NSE \eqref{NS001} in $C^{\infty}(\mathbb{R}^{3}\setminus\{x'=0'\})$. Here and below, we represent the two-dimensional variables by adding superscript prime such as $x'=(x_{1},x_{2})$. According to the singularities, these solutions can be divided into the following three types:
\begin{itemize}
{\it
\item[] Type 1. Landau solutions satisfying $0<\limsup\limits_{|x|\rightarrow0}|x||u^{c,\gamma}|<\infty$;
\item[] Type 2. Solutions satisfying $0<\limsup\limits_{|x|=1,x'\rightarrow0'}(-\ln|x'|)^{-1}|u^{c,\gamma}|<\infty$;
\item[] Type 3. Solutions satisfying $0<\limsup\limits_{|x|=1,x'\rightarrow0'}|x'||u^{c,\gamma}|<\infty$.}
\end{itemize}
Remark that if $|c|=0$ and $\gamma\neq0$, $u^{c,\gamma}$ becomes Landau solution, see Theorem B in \cite{LLY201801}. The second and third types of solutions exhibit anisotropic singularities around all the $x_{3}$-axis. This is different from Landau solutions \cite{L1944}, whose singularity is isotropic and only concentrates in small neighbourhood near the origin. With regard to Landau solutions, Tian and Xin \cite{TX1998} showed that all $(-1)$-homogeneous, axisymmetric nonzero solutions of \eqref{NS001} in $C^{\infty}(\mathbb{R}^{3}\setminus\{0\})$ are Landau solutions. \v{S}ver\'{a}k \cite{S2011} proved that all $(-1)$-homogeneous nontrivial smooth solutions of \eqref{NS001} in $C^{\infty}(\mathbb{R}^{3}\setminus\{0\})$ are Landau solutions. In recent years, Li, Li and Yan \cite{LLY201801,LLY201802,LLY2019} systematically studied the $(-1)$-homogeneous axisymmetric solutions of \eqref{NS001} in $C^{\infty}(\mathbb{R}^{3}\setminus\{x'=0'\})$ with a possible singular ray $\{x'=0'\}$. To be specific, they gave a complete classification for such solutions with no-swirl in \cite{LLY201801,LLY201802} and showed the existence of such solutions with nonzero swirl in \cite{LLY201801,LLY2019}. For more earlier work on $(-1)$-homogeneous solutions, we refer to \cite{G1960,PP198501,PP198502,PP198503,S1972,S1951,W1991,Y1953} and the references therein.

In this paper, we focus on the second type of singular solutions, that is, the least singular solutions among these solutions other than Landau solutions. Denote $M:=\{(c,\gamma)|\,c_{1}=c_{2}=0,c_{3}>-4,\gamma^{+}(c)>\gamma>\gamma^{-}(c)\}$. Li and Yan \cite{LY2021} proved that there exists a small constant $\mu_{0}>0$ such that for any $(c,\gamma)\in M\cap\{|(c,\gamma)|<\mu_{0}\}$, $(u^{c,\gamma},p^{c,\gamma})$ is a weak solution of the stationary Navier-Stokes equations with singular force as follows:
\begin{align}\label{NS003}
\begin{cases}
-\Delta u^{c,\gamma}+u^{c,\gamma}\cdot\nabla u^{c,\gamma}+\nabla p^{c,\gamma}=f^{c,\gamma},\\
\mathrm{div}u^{c,\gamma}=0,
\end{cases}
\end{align}
where $f^{c,\gamma}=(4\pi c_{3}\ln|x_{3}|\partial_{x_{3}}\delta_{(0,0,x_{3})}-b^{c,\gamma}\delta_{0})e_{3}$, $e_{3}=(0,0,1)$, $\delta_{0}$ is the Dirac measure, and $b^{c,\gamma}$ is given by
\begin{align*}
b^{c,\gamma}=\int^{1}_{-1}\left(y[(U^{c,\gamma}_{\theta})']^{2}-\frac{2-y^{2}}{1-y^{2}}U^{c,\gamma}_{\theta}-\frac{y}{1-y^{2}}(U_{\theta}^{c,\gamma})^{2}\right)\rightarrow0,\quad\text{as }|(c,\gamma)|\rightarrow0.
\end{align*}
The weak solution of \eqref{NS003} are understood in the sense that for every test function $\varphi\in C^{\infty}_{c}(\mathbb{R}^{3})$,
\begin{align*}
&\int_{\mathbb{R}^{3}}(\nabla u_{j}^{c,\gamma}\nabla\varphi-u^{c,\gamma}_{j}u^{c,\gamma}\cdot\nabla\varphi-p^{c,\gamma}\partial_{x_{j}}\varphi)\notag\\
&=
\begin{cases}
0,&j=1,2,\\
4\pi c_{3}\int^{\infty}_{-\infty}\ln|x_{3}|\partial_{x_{3}}\varphi(0,0,x_{3})dx_{3}-b^{c,\gamma}\varphi(0),&j=3,
\end{cases}
\end{align*}
and
\begin{align*}
\int_{\mathbb{R}^{3}}u^{c,\gamma}\cdot\nabla\varphi=0.
\end{align*}
From Corollary 2.1 in \cite{LY2021}, we obtain that for $x\in\mathbb{R}^{3}\setminus\{x'=0'\}$,
\begin{align}\label{Singular}
|u^{c,\gamma}|=\frac{2|c_{3}|}{|x|}\ln\frac{|x|}{|x'|}+O(1)\frac{|(c,\gamma)|}{|x|},\quad|\nabla u^{c,\gamma}|=\frac{2|c_{3}|}{|x'||x|}+\frac{O(1)|(c,\gamma)|}{|x|^{2}}\ln\frac{|x|}{|x'|}.
\end{align}

Let $\dot{H}^{1}(\mathbb{R}^{3})$ be the closure of $C^{\infty}_{c}(\mathbb{R}^{3};\mathbb{R}^{3})$ under the norm $\|\nabla u\|_{L^{2}(\mathbb{R}^{3})}$. For any $1\leq p<\infty$, introduce the following spaces:
\begin{align*}
L_{\sigma}^{p}(\mathbb{R}^{3})=\{u\in L^{p}(\mathbb{R}^{3})|\,\nabla\cdot u=0\},\quad \dot{H}^{1}_{\sigma}(\mathbb{R}^{3})=\{u\in\dot{H}^{1}(\mathbb{R}^{3})|\,\nabla\cdot u=0\},
\end{align*}
with their norms as
\begin{align*}
\|u\|_{L^{p}_{\sigma}(\mathbb{R}^{3})}:=\|u\|_{L^{p}(\mathbb{R}^{3})},\quad\|u\|_{\dot{H}^{1}_{\sigma}(\mathbb{R}^{3})}=\|\nabla u\|_{L^{2}(\mathbb{R}^{3})}.
\end{align*}
For $w_{0}\in L^{p}_{\sigma}(\mathbb{R}^{3})$, consider the Cauchy problem for the incompressible Navier-Stokes equations in $\mathbb{R}^{3}\times(0,\infty)$ as follows:
\begin{align*}
\begin{cases}
\partial_{t}u-\Delta u+u\cdot\nabla u+\nabla p=f^{c,\gamma},\\
\mathrm{div}u=0,\\
u(x,0)=u^{c,\gamma}+w_{0},
\end{cases}
\end{align*}
where $u^{c,\gamma}$ and $f^{c,\gamma}$ are given by \eqref{NS003}. Denote $w(x,t)=u(x,t)-u^{c,\gamma}(x)$ and $\pi(x)=p(x)-p^{c,\gamma}(x)$. Then $(w(x,t),\pi(x))$ satisfies the following perturbed equations
\begin{align}\label{NS005}
\begin{cases}
\partial_{t}w-\Delta w+w\cdot\nabla w+w\cdot\nabla u^{c,\gamma}+u^{c,\gamma}\cdot\nabla w+\nabla \pi=0,\\
\mathrm{div}w=0,\\
w(x,0)=w_{0}.
\end{cases}
\end{align}
Li and Yan \cite{LY2021} obtained the asymptotic stability for the solution $w$ of problem \eqref{NS005} by using the same arguments in \cite{KP2011}, where Karch and Pilarczyk \cite{KP2011} showed that small Landau solutions keep asymptotically stable under $L^{2}$-perturbations. To be specific, by utilizing \eqref{Singular} and the anisotropic Caffarelli-Kohn-Nirenberg inequalities established in \cite{LY202102}, the following crucial fact
\begin{align*}
\left|\int_{\mathbb{R}^{3}}(w\cdot\nabla u^{c,\gamma})\cdot w\right|
\leq& K(c,\gamma)\|\nabla w\|^{2}_{L^{2}},\;\,\mathrm{with}\;K(c,\gamma)\rightarrow0,\;\,\mathrm{as}\;|(c,\gamma)|\rightarrow0
\end{align*}
holds, then the proofs for asymptotic stability of the Landau solutions in \cite{KP2011} can be directly applied to problem \eqref{NS005}. In 2017, Karch, Pilarczyk and Schonbek \cite{KPS2017} further generalized the work in \cite{KP2011} and developed a new method which allows to establish $L^{2}$-asymptotic stability for a large class of global-in-time solutions, especially covering the Landau solutions. Moreover, their results also generalize a series of works concerning the $L^{2}$-asymptotic stability either of the zero solution \cite{W1987,BM1992,KM1986,ORS1997,S1980,S1985} or nontrivial stationary solutions \cite{BM1995} to the Navier-Stokes equations.

Define the linear operator as follows:
\begin{align}\label{OPE001}
\mathcal{L}w=-\Delta w+\mathbb{P}((w\cdot\nabla)u^{c,\gamma})+\mathbb{P}((u^{c,\gamma}\cdot\nabla)w),
\end{align}
where $\mathbb{P}$ is the Leray projector onto the divergence-free vector fields. With regard to the properties of the operator $\mathcal{L}$, when $u^{c,\gamma}$ is replaced with Landau solution in \eqref{OPE001}, Karch and Pilarczyk \cite{KP2011} proved that $-\mathcal{L}$ is the infinitesimal generator of an analytic semigroup of bounded linear operators on $L^{2}_{\sigma}(\mathbb{R}^{3})$. Recently, Li, Zhang and Zhang \cite{LZZ2023} further extend the result on $L^{q}_{\sigma}(\mathbb{R}^{3})$ with $1<q<\infty$. Making use of \eqref{OPE001} and the Duhamel principle, the solution $w$ of equations \eqref{NS005} can be rewritten as the following integral form
\begin{align}\label{INT001}
w(x,t)=e^{-t\mathcal{L}}w_{0}-\int^{t}_{0}e^{-(t-s)\mathcal{L}}\mathbb{P}\mathrm{div}(w\otimes w)ds:=a+N(w,w).
\end{align}
Combining \eqref{Singular} and the anisotropic Caffarelli-Kohn-Nirenberg inequalities in \cite{LY202102}, it follows from the proof of Theorem 3.1 in \cite{LZZ2023} with a slight modification that $e^{-t\mathcal{L}}$ is also an analytic semigroup of bounded linear operators on $L^{q}_{\sigma}(\mathbb{R}^{3})$ for any $1<q<\infty$. From the work of Fabes-Jones-Rivi\`{e}re \cite{FJR1972}, we see that the solutions of the integral form in \eqref{INT001} are equivalent to the very weak solutions due to the minimal regularity assumption of only being in $L_{t}^{q}L_{x}^{p}$ with $p,q\geq2$. These solutions are called \textit{mild solutions}.

In this paper, we aim to make clear the asymptotic stability of the second type of singular solutions to Navier-Stokes equations under $L^{3}$-perturbations. For this purpose, we should first handle the well-posedness problem for solutions to the perturbed equations \eqref{NS005}. The method used in this paper is based on the perturbation theory, which has been presented in recent work \cite{LZZ2023} for the Landau solutions. The core idea lies in treating the nonlinear part $N(w,w)$ as a perturbation of the linear part $a$ in \eqref{INT001} and solving these two parts by successive approximation with strong convergence. Then combining the contraction mapping principle, we obtain the global well-posedness results for $L^{3}$ mild solutions with small initial data in $L_{\sigma}^{3}$ and the local well-posedness of $L^{p}$ mild solutions (see Definition \ref{DEF69}) under any initial data in $L_{\sigma}^{p}$ for $p\geq3$. These mild solutions can be actually regarded as a special family of strong solutions due to their uniqueness and better regularity.

As a continuation of \cite{LY2021,LZZ2023}, the results of this paper are not trivial generalizations. On one hand, by contrast with \cite{LZZ2023}, we will give a more clear understanding on application of the perturbation method in addressing the well-posedness problem for the solutions by providing more complete and strict proofs. On the other hand, we capture a precise constant in the asymptotic stability result of Theorem \ref{THM003} by introducing a special accuracy parameter $\tau$, which can be regarded as a threshold value for controlling the disturbance effect arising from singular solution $u^{c,\gamma}$. In fact, from the view of the structure of equations \eqref{NS005}, the difference between the perturbed equations \eqref{NS005} and Navier-Stokes equations lies in that there are two more disturbance terms $w\cdot\nabla u^{c,\gamma}$ and $u^{c,\gamma}\cdot\nabla w$ appearing in the first equation of \eqref{NS005}. The effect caused by these two disturbance terms can be weakened by decreasing the value of $|(c,\gamma)|$ in \eqref{Singular}. In addition, we will prove in Section \ref{SEC005} below that the unique $L^{p}$ mild solutions of \eqref{NS005} with the initial data in $L^{p}_{\sigma}\cap L_{\sigma}^{2}(\mathbb{R}^{3})$ are also the $L^{2}$ weak solutions (see Definition \ref{DEFINI06}). This fact not only indicates that this paper provides an alternative proof for the results in \cite{LY2021}, but also ensures the uniqueness either of the global $L^{2}$ weak solutions under small initial data or the local $L^{2}$ weak solutions with all initial data in $L_{\sigma}^{p}\cap L_{\sigma}^{2}(\mathbb{R}^{3})$ for $p\geq3$.

Aside from the perturbation theory, another classical method for studying the existence of solutions is the energy method. The energy method is based on the establishment of a priori estimate which is used to construct large and global-in-time weak solutions by taking weak limits. The existence of weak solutions to the Navier-Stokes equations has been known for quite long time since the great work \cite{L1934} of Leray, where in \cite{L1934} Leray constructed weak solutions $u\in L_{t}^{\infty}L^{2}\cap L_{t}^{2}H^{1}$ with divergence-free initial data $u_{0}\in L^{2}(\mathbb{R}^{3})$. These solutions are called Leray-Hopf weak solutions and satisfy the energy inequality. Subsequently, Hopf \cite{H1951} obtained a similar result for the equations in a smooth bounded domain with Dirichlet boundary condition. However, the problems of uniqueness and regularity of Leray-Hopf weak solutions in dimensions greater than two remain open and are regarded as one of the most important issues in fluid mechanics. Recently, Buckmaster and Vicol \cite{BV2019} utilized the method of convex integration to establish the nonuniqueness of weak solution for the three-dimensional Navier-Stokes equations with finite energy. It is worth pointing out that the weak solutions constructed in \cite{BV2019} are not known to be of Leray-Hopf. As for the nonuniqueness of weak solutions for the stationary Navier-Stokes, see \cite{L2019}. In addition, the convex integration scheme has already been developed to prove the nonuniqueness of solutions for other PDEs in fluid dynamics, see \cite{BLSV2018,BSV2019} and the references therein.

To state our results in a precise manner, we now give the precise definition of $L^{p}$ mild solution for the perturbed problem \eqref{NS005}.
\begin{definition}\label{DEF69}
Set $3\leq p<\infty$ and $T>0$. For any given initial data $w_{0}\in L^{p}(\mathbb{R}^{3})$, a function $w$ is called a $L^{p}$ mild solution of equations \eqref{NS005} on $[0,T]$, provided
\begin{align}\label{AQ001}
w\in C([0,T];L_{\sigma}^{p}(\mathbb{R}^{3}))\cap L^{\frac{4p}{3}}([0,T];L_{\sigma}^{2p}(\mathbb{R}^{3})),
\end{align}
and
\begin{align}\label{AQ002}
w(x,t)=e^{-t\mathcal{L}}w_{0}-\int^{t}_{0}e^{-(t-s)\mathcal{L}}\mathbb{P}\mathrm{div}(w\otimes w)ds.
\end{align}
Further, this solution is global provided \eqref{AQ001}--\eqref{AQ002} hold for any $0<T<\infty$.
\end{definition}

To begin with, we list the local well-posedness results for $L^{p}$ mild solutions to the perturbed problem \eqref{NS005} with any $p\geq3$ as follows.
\begin{theorem}\label{THM001}
For $p\geq3$ and $w_{0}\in L_{\sigma}^{p}(\mathbb{R}^{3})$, there exist two small positive constants $\delta=\delta(p)$ and $T=T(p,\|w_{0}\|_{L^{p}})$ such that for any $(c,\gamma)\in M\cap\{|(c,\gamma)|\leq\delta\}$, problem \eqref{NS005} possesses a unique $L^{p}$ mild solution $w$ on $[0,T]$ with $\nabla|w|^{\frac{p}{2}}\in L^{2}([0,T];L^{2}(\mathbb{R}^{3}))$. Moreover,
\begin{align}\label{AQ005}
\|w\|_{C_{T}L_{x}^{p}\cap L_{T}^{\frac{4p}{3}}L_{x}^{2p}}+\|\nabla|w|^{\frac{p}{2}}\|^{\frac{2}{p}}_{L^{2}_{T}L^{2}_{x}}\leq C(p,c,\gamma)\|w_{0}\|_{L^{p}(\mathbb{R}^{3})}.
\end{align}

\end{theorem}
\begin{remark}
As seen in Theorem \ref{THM001}, the existence and uniqueness of local $L^{p}$ mild solution $w$ have been established. With regard to the continuous dependence of the solution $w$ on the initial data $w_{0}$, using the same proof as in Theorem 1.3 of \cite{LZZ2023}, we obtain that there exists a small $\varepsilon>0$ such that for any $v_{0}\in L_{\sigma}^{p}(\mathbb{R}^{3})$, if $\|w_{0}-v_{0}\|_{L^{p}}<\varepsilon$, then there exists a unique $L^{p}$ mild solution $v$ on $[0,T]$ with the initial data $v_{0}$. Furthermore,
\begin{align*}
\|w-v\|_{C_{T}L_{x}^{p}\cap L_{T}^{\frac{4p}{3}}L_{x}^{2p}}+\big\|\nabla|w-v|^{\frac{p}{2}}\big\|^{\frac{2}{p}}_{L_{T}^{2}L_{x}^{2}}\rightarrow0,\quad\text{as $\|w_{0}-v_{0}\|_{L^{p}}\rightarrow0.$}
\end{align*}

\end{remark}
\begin{remark}
With regard to the method of space-time estimates applied to the well-posedness problem for more general dissipative equations including the Navier-Stokes equations, see \cite{MYZ2008,MZ2004,M2001}.

\end{remark}

Second, the global well-posedness results under small initial data are stated as follows.
\begin{theorem}\label{THM002}
For $p\in[3,\frac{9}{2}]$, there exist two small positive constants $\varepsilon_{0}$ and $\delta_{0}$ such that for any $(c,\gamma)\in M\cap\{|(c,\gamma)|\leq\delta_{0}\}$ and $w_{0}\in L^{p}_{\sigma}(\mathbb{R}^{3})\cap L^{3}_{\sigma}(\mathbb{R}^{3})\cap \{\|w_{0}\|_{L^{3}(\mathbb{R}^{3})}<\varepsilon_{0}\}$, problem \eqref{NS005} has a unique global $L^{p}$ mild solution $w$ with $\nabla|w|^{\frac{p}{2}}\in L^{2}([0,\infty);L^{2}(\mathbb{R}^{3}))$. Furthermore, \eqref{AQ005} holds with the integrating range  $[0,T]$ replaced by $[0,\infty)$.

\end{theorem}

\begin{remark}
Although it remains to be open whether the the global existence result for $L^{3}$ mild solution under small initial data will imply its global existence for all initial data, we can find a global $L^{2}+L^{3}$ weak solution for any initial data in $L_{\sigma}^{3}(\mathbb{R}^{3})$ by using an idea in \cite{LZZ2023}, which is actually inspired by previous work \cite{C1990,KPS2017,SS2017}. Specifically speaking, we first decompose the initial data $w_{0}\in L_{\sigma}^{3}(\mathbb{R}^{3})$ into two parts as follows: $w_{0}=w_{01}+w_{02}$, where $w_{01}\in L^{3}_{\sigma}(\mathbb{R}^{3})$ satisfies $\|w_{01}\|_{L^{3}}<\varepsilon_{0}$ and $w_{02}\in L_{\sigma}^{2}\cap L_{\sigma}^{3}(\mathbb{R}^{3})$. By using the perturbation theory, we obtain a unique global $L^{3}$ mild solution $w_{1}$ for problem \eqref{NS005} with the small initial data $w_{01}$. Then we proceed to make use of the energy method to construct a weak solution $w_{2}\in C_{w}([0,\infty);L^{2}_{\sigma}(\mathbb{R}^{3}))\cap L^{2}([0,\infty);\dot{H}_{\sigma}^{1}(\mathbb{R}^{3}))$ satisfying the following equations
\begin{align*}
\begin{cases}
\partial_{t}w_{2}-\Delta w_{2}+w_{2}\cdot\nabla w_{2}+w_{2}\cdot\nabla(u^{c,\gamma}+w_{1})+(u^{c,\gamma}+w_{1})\cdot\nabla w_{2}+\nabla \pi_{2}=0,\\
\mathrm{div}w_{2}=0,\\
w_{2}(x,0)=w_{02}.
\end{cases}
\end{align*}
Then $w:=w_{1}+w_{2}$ forms a global $L^{2}+L^{3}$ weak solution of the perturbed problem \eqref{NS005} which consists of a $L^{3}$ mild solution $w_{1}\in C_{t}L_{x}^{3}\cap L_{t}^{4}L_{x}^{6}$ with the small initial data $w_{01}\in L^{3}_{\sigma}$ and a $L^{2}$ weak solution $w_{2}\in C_{w}L_{x}^{2}\cap L_{t}^{2}\dot{H}_{x}^{1}$ with the initial data $w_{02}\in L^{2}_{\sigma}\cap L_{\sigma}^{3}$.

\end{remark}

Based on the global well-posedness results obtained in Theorem \ref{THM002}, we further study the asymptotic stability of the global $L^{3}$ mild solution for problem \eqref{NS005} and the corresponding results are given as follows.
\begin{theorem}\label{THM003}
Let $\varepsilon_{0}$ and $\delta_{0}$ be given in Theorem \ref{THM002} and $w$ be the unique global $L^{3}$ mild solution for problem \eqref{NS005} with $w_{0}\in L^{3}_{\sigma}(\mathbb{R}^{3})$. Then

$(i)$ for $q=3,$ if $(c,\gamma)\in M\cap\{|(c,\gamma)|\leq\delta_{0}\}$ and $\|w_{0}\|_{L^{3}(\mathbb{R}^{3})}<\frac{\varepsilon_{0}}{2}$, we obtain that
\begin{align*}
\lim\limits_{t\rightarrow\infty}\|w(t)\|_{L^{3}(\mathbb{R}^{3})}=0;
\end{align*}

$(ii)$ for $q>3$ and any $0<\tau<1$, there exists two small positive constants $\delta=\delta(\tau,q)\leq\delta_{0}$ and $\varepsilon=\varepsilon(\tau,q)\leq\varepsilon_{0}$ such that if $(c,\gamma)\in M\cap\{|(c,\gamma)|\leq\delta\}$ and $\|w_{0}\|_{L^{3}(\mathbb{R}^{3})}<\varepsilon$, we have
\begin{align}\label{AS06}
\|w(t)\|_{L^{q}(\mathbb{R}^{3})}\leq \mathcal{C}_{q}\left(\frac{1}{3}-\frac{1}{q}\right)^{\frac{3}{2}(\frac{1}{3}-\frac{1}{q})}t^{-\frac{3}{2}(\frac{1}{3}-\frac{1}{q})}\|w_{0}\|_{L^{3}(\mathbb{R}^{3})},\quad\text{for all}\;t>0,
\end{align}
where
\begin{align}\label{CON001}
\mathcal{C}_{q}=3^{-\frac{7}{4}}q^{\frac{3}{q}}\left(q-2\right)^{\frac{3}{2q}}\left(\frac{q}{q-2}\right)^{\frac{3}{4}}(4\pi(1-\tau))^{-\frac{3}{2}(\frac{1}{3}-\frac{1}{q})}e^{-6(\frac{1}{3}-\frac{1}{q})}.
\end{align}
\end{theorem}
\begin{remark}
In Theorem \ref{THM003}, the parameter $\tau$ is introduced to quantitatively describe the disturbance effect produced by these two terms $w\cdot\nabla u^{c,\gamma}$ and $u^{c,\gamma}\cdot\nabla w$ in equations \eqref{NS005} and confine their disturbance effect to be a given range determined by $\tau$. For any given accuracy parameter $0<\tau<1$, we can find two sufficiently small positive constants $\delta(\tau,q)$ and $\varepsilon(\tau,q)$ to achieve a precise calculation for the value of $\mathcal{C}_{q}$. According to \eqref{CON001}, we see that the constant $\mathcal{C}_{q}$ is smooth in $q$, and
\begin{align*}
\lim_{q\rightarrow3^{+}}\mathcal{C}_{q}=1,\quad\lim_{q\rightarrow+\infty}\mathcal{C}_{q}=\frac{1}{3^{\frac{7}{4}}e^{2}\sqrt{4\pi(1-\tau)}}.
\end{align*}
This fact also implies that the constant $\mathcal{C}_{q}$ is sharp in the sense that there is equality in \eqref{AS06} by first sending $q\rightarrow3^{+}$ and then $t\rightarrow0^{+}$. Based on these above facts, we believe that this precise constant will play an important role in numerically analyzing and simulating the properties of asymptotic stability of the solutions in future work.

%Our calculations in Theorem \ref{THM003} improves the computational result obtained in Theorem 1.1 of \cite{LZZ2023}.
\end{remark}

\begin{remark}
According to the global existence results in Theorem \ref{THM002}, we can extend the stability results in Theorem \ref{THM003} to the case of $p\in(3,\frac{9}{2}]$ under a stronger assumed condition of $w_{0}\in L^{p}_{\sigma}(\mathbb{R}^{3})\cap L^{3}_{\sigma}(\mathbb{R}^{3})\cap \{\|w_{0}\|_{L^{3}(\mathbb{R}^{3})}<\varepsilon\}$ for some sufficiently small positive constant $\varepsilon.$ In fact, since the global existence result of $3<p\leq\frac{9}{2}$ is a direct consequence of the global existence result of $p=3$ together with the Gagliardo-Nirenberg interpolation inequality, then we here place main emphasis on the asymptotic stability problem in the case of $p=3$.

\end{remark}

The paper is organized as follows. In Section \ref{SEC002}, we do some preliminary work and list some results which will be used later. Section \ref{SEC003} is devoted to solving the well-posedness of $L^{p}$ mild solutions to the perturbed equations \eqref{NS005}. Then we study the asymptotic stability of the global $L^{3}$ mild solutions with small initial data in Section \ref{SEC004}. Finally, we make clear the relations between $L^{p}$ mild solutions and $L^{2}$ weak solutions for the perturbed problem \eqref{NS005} in Section \ref{SEC005}.

\section{Preliminaries}\label{SEC002}

In this section, we mainly state some results which will be used in the following proofs. As pointed out in the introduction, we will make use of the anisotropic Caffarelli-Kohn-Nirenberg inequalities obtained in \cite{LY202102} to deal with the terms involving singular solution $u^{c,\gamma}$. For readers' convenience, we list the anisotropic Caffarelli-Kohn-Nirenberg inequalities in \cite{LY202102} as follows. For $n\geq2$, let $\theta$ and $s_{i},\alpha_{i},\beta_{i}$, $i=1,2,3$ be real numbers satisfying that
\begin{itemize}
{\it
\item[]
\begin{align}\label{ME001}
s_{1},s_{3}>0,\quad s_{2}\geq1,\quad0\leq \theta \leq1,
\end{align}
\item[]
\begin{align}\label{ME002}
\begin{cases}
\frac{1}{s_{i}}+\frac{\alpha_{i}}{n-1}>0,\\
\beta_{i}\geq0,
\end{cases}
\quad\mathrm{or}\quad
\begin{cases}
\frac{1}{s_{i}}+\frac{\alpha_{i}}{n-1}>0,\\
\beta_{i}<0,\\
\frac{1}{s_{i}}+\frac{\alpha_{i}+\beta_{i}}{n}>0,
\end{cases}\quad i=1,2,3,
\end{align}
\item[]
\begin{align}\label{ME003}
\frac{1}{s_{1}}+\frac{\alpha_{1}+\beta_{1}}{n}=\theta\left(\frac{1}{s_{2}}+\frac{\alpha_{2}+\beta_{2}-1}{n}\right)+(1-\theta)\left(\frac{1}{s_{3}}+\frac{\alpha_{3}+\beta_{3}}{n}\right),  \end{align}
\item[]
\begin{align}\label{ME005}
\beta_{1}\leq\theta\beta_{2}+(1-\theta)\beta_{3},
\end{align}
\item[]
\begin{align}\label{ME006}
\alpha_{1}+\beta_{1}\leq\theta(\alpha_{2}+\beta_{2})+(1-\theta)(\alpha_{3}+\beta_{3}),
\end{align}
\item[]
\begin{align}\label{ME007}
\frac{1}{s_{1}}+\frac{\alpha_{1}}{n-1}\geq\theta\left(\frac{1}{s_{2}}+\frac{\alpha_{2}-1}{n-1}\right)+(1-\theta)\left(\frac{1}{s_{3}}+\frac{\alpha_{3}}{n-1}\right),  \end{align}
\item[]
\begin{align}\label{ME008}
&\frac{1}{s_{1}}\leq\frac{\theta}{s_{2}}+\frac{1-\theta}{s_{3}},\;\,\mathrm{if}\;\frac{1}{s_{1}}+\frac{\alpha_{1}+\beta_{1}}{n}=\frac{1}{s_{2}}+\frac{\alpha_{2}+\beta_{2}-1}{n}=\frac{1}{s_{3}}+\frac{\alpha_{3}+\beta_{3}}{n},\notag\\
&\mathrm{or}\;\theta=0\;\mathrm{or}\;\theta=1\;\mathrm{or}\;\frac{1}{s_{1}}+\frac{\alpha_{1}}{n-1}=\theta\left(\frac{1}{s_{2}}+\frac{\alpha_{2}-1}{n-1}\right)+(1-\theta)\left(\frac{1}{s_{3}}+\frac{\alpha_{3}}{n-1}\right).
\end{align}

    }
\end{itemize}

Then we have
\begin{lemma}[Theorem 1.1 in \cite{LY202102}]\label{ACKN001}
For $n\geq2$, let $\theta$ and $s_{i},\alpha_{i},\beta_{i}$, $i=1,2,3$ be real numbers satisfying conditions \eqref{ME001}--\eqref{ME002}. Then there exists some positive constant $C=C(s_{1},s_{2},s_{3},\alpha_{1},\alpha_{2},\alpha_{3},\beta_{1},\beta_{2},\beta_{3},\theta)$ such that
\begin{align*}
\||x'|^{\alpha_{1}}|x|^{\beta_{1}}u\|_{L^{s_{1}}(\mathbb{R}^{n})}\leq C\||x'|^{\alpha_{2}}|x|^{\beta_{2}}u\|^{\theta}_{L^{s_{2}}(\mathbb{R}^{n})}\||x'|^{\alpha_{3}}|x|^{\beta_{3}}u\|^{1-\theta}_{L^{s_{3}}(\mathbb{R}^{n})}
\end{align*}
holds for any $u\in C_{c}^{\infty}(\mathbb{R}^{n})$ if and only if \eqref{ME003}--\eqref{ME008} hold. Moreover, on any compact subset of the parameter range in which \eqref{ME001}--\eqref{ME002} hold, $C$ is a bounded constant.
\end{lemma}

\begin{remark}
Condition \eqref{ME002} is called measure condition such that
$$\||x'|^{\alpha_{i}}|x|^{\beta_{i}}u\|_{L^{s_{i}}(\mathbb{R}^{n})}<\infty,\quad i=1,2,3,\;\text{for any $u\in C_{c}^{\infty}(\mathbb{R}^{n})$},$$
which can be obtained by using Lemma 2.1 in \cite{MZ2022}.
\end{remark}
A direct application of Lemma \ref{ACKN001} gives the following corollary.
\begin{corollary}\label{MZLEM001}
For any $0\leq\alpha<1$ and $u\in C_{c}^{\infty}(\mathbb{R}^{3})$, we have
\begin{align*}
\||x'|^{-\alpha}|x|^{\alpha-1}u\|_{L^{2}(\mathbb{R}^{3})}\leq C(\alpha)\|\nabla u\|_{L^{2}(\mathbb{R}^{3})}.
\end{align*}

\end{corollary}
\begin{remark}
As seen in Corollary \ref{MZLEM001}, we restrict the range of $\alpha$ to be in $[0,1)$ and exclude the case of $\alpha=1$, since the case of $\alpha=1$ don't satisfy measure condition \eqref{ME002}.
\end{remark}

For $\rho>0$, define a cone as follow:
\begin{align}\label{CONE999}
\Omega_{\rho}:=\{x\in\mathbb{R}^{3}\,|\,|x'|\leq\rho|x|\}.
\end{align}
From \eqref{Singular}, we obtain
\begin{lemma}\label{CORO001}
Let $u^{c,\gamma}$ be given in \eqref{Singular}. Then we have
\begin{align}\label{Z01}
\begin{cases}
\||x'|^{\frac{1}{2}}|x|^{\frac{1}{2}}u^{c,\gamma}\|_{L^{\infty}(\Omega_{e^{-1}})}\leq K(c,\gamma),\\
\||x|u^{c,\gamma}\|_{L^{\infty}(\mathbb{R}^{3}\setminus\Omega_{e^{-1}})}\leq K(c,\gamma),
\end{cases}
\end{align}
and
\begin{align}\label{Z02}
\||x'||x|\nabla u^{c,\gamma}\|_{L^{\infty}(\mathbb{R}^{3})}\leq K(c,\gamma),
\end{align}
where $\Omega_{e^{-1}}$ is a cone defined by \eqref{CONE999}, $K(c,\gamma)$ satisfies that $K(c,\gamma)\rightarrow0$, as $|(c,\gamma)|\rightarrow0$.

\end{lemma}
\begin{proof}
To begin with, for any $x\in\Omega_{e^{-1}}$, we have $-\ln\frac{|x'|}{|x|}\geq1$. This, together with the fact that $|t^{\alpha}\ln t|\leq (\alpha e)^{-1}$ in $[0,1]$ with $0<\alpha\leq1$, shows that
\begin{align}\label{PU09}
\frac{1}{|x|}\leq-\frac{1}{|x|}\ln\frac{|x'|}{|x|}\leq \frac{2}{e}|x'|^{-\frac{1}{2}}|x|^{-\frac{1}{2}}.
\end{align}
Combining \eqref{Singular} and \eqref{PU09}, we deduce
\begin{align*}
\||x'|^{\frac{1}{2}}|x|^{\frac{1}{2}}u^{c,\gamma}\|_{L^{\infty}(\Omega_{e^{-1}})}\leq K(c,\gamma).
\end{align*}
On the other hand, for any $x\in \mathbb{R}^{3}\setminus\Omega_{e^{-1}}$, we have $-\ln\frac{|x'|}{|x|}\leq1$. It then follows from \eqref{Singular} that
\begin{align*}
\||x|u^{c,\gamma}\|_{L^{\infty}(\mathbb{R}^{3}\setminus\Omega_{e^{-1}})}\leq K(c,\gamma).
\end{align*}
Finally, since
\begin{align*}
-\frac{1}{|x|^{2}}\ln\frac{|x'|}{|x|}\leq\frac{1}{e}\frac{1}{|x'||x|},\quad\text{for any $x\in\mathbb{R}^{3}$},
\end{align*}
then we have from \eqref{Singular} that \eqref{Z02} holds. The proof is complete.

\end{proof}

\begin{definition}
Let $1<q<\infty$ and $n\geq2$. We say that $w$ is an $A_{q}$-weight, if there is a positive constant $C=C(n,q,w)$ such that
\begin{align*}
\dashint_{B}wdx\left(\dashint_{B}w^{-\frac{1}{q-1}}dx\right)^{q-1}\leq C(n,q,w),\quad\mathrm{with}\;\dashint_{B}=\frac{1}{|B|}\int_{B},
\end{align*}
for any ball $B$ in $\mathbb{R}^{n}$.
\end{definition}
For $n\geq2$ and $q>1$, define the following indexing sets:
\begin{align*}
\begin{cases}
\mathcal{A}=\{(\theta_{1},\theta_{2}): \theta_{1}>-(n-1),\,\theta_{2}\geq0\},\\
\mathcal{B}=\{(\theta_{1},\theta_{2}):\theta_{1}>-(n-1),\,\theta_{2}<0,\,\theta_{1}+\theta_{2}>-n\},\\
\mathcal{C}_{q}=\{(\theta_{1},\theta_{2}):\theta_{1}<(n-1)(q-1),\,\theta_{2}\leq0\},\\
\mathcal{D}_{q}=\{(\theta_{1},\theta_{2}):\theta_{1}<(n-1)(q-1),\,\theta_{2}>0,\,\theta_{1}+\theta_{2}<n(q-1)\}.
\end{cases}
\end{align*}
Recall Theorem 2.6 in \cite{MZ2022} as follows.
\begin{lemma}[Theorem 2.6 in \cite{MZ2022}]\label{MZLEM002}
Let $1<q<\infty$ and $n\geq2$. If $(\theta_{1},\theta_{2})\in(\mathcal{A}\cup\mathcal{B})\cap(\mathcal{C}_{q}\cup\mathcal{D}_{q})$, then $w=|x'|^{\theta_{1}}|x|^{\theta_{2}}$ is an $A_{q}$-weight.
\end{lemma}

We now introduce the contraction mapping theorem (see Lemma 5.5 in \cite{BCD2011}), which is the core tool to application of the perturbation theory.
\begin{lemma}\label{PER001}
Let $E$ be a Banach space, $N$ be a continuous bilinear map from $E\times E$ to $E$, and $\alpha$ be a positive real number satisfying that
\begin{align*}
\alpha<\frac{1}{4\|N\|},\quad\|N\|:=\sup\limits_{\|u\|,\|v\|\leq1}\|N(u,v)\|.
\end{align*}
Then for any $a\in B(0,\alpha)$ (i.e., with center $0$ and radius $\alpha$) in $E$, there exists a unique $x\in B(0,2\alpha)$ such that
\begin{align*}
x=a+N(x,x).
\end{align*}

\end{lemma}

\section{Well-posedness for solutions to the perturbed equations}\label{SEC003}

Before using Lemma \ref{PER001} to prove Theorems \ref{THM001} and \ref{THM002}, we need first study the well-posedness for the linear part $a$ and the nonlinear part $N$ in $\eqref{INT001}$, respectively. From Theorem 2.1 in \cite{FJR1972}, we see that the linear part $a$ is a solution of
\begin{align}\label{LINE001}
\begin{cases}
\partial_{t}a-\Delta a+(a\cdot\nabla)u^{c,\gamma}+(u^{c,\gamma}\cdot\nabla)a+\nabla\pi_{1}=0,\\
\mathrm{div}a=0,\\
a(x,0)=w_{0}(x).
\end{cases}
\end{align}
That is, for any $\varphi\in C^{\infty}_{c}([0,\infty)\times\mathbb{R}^{3})$ with $\mathrm{div}\varphi=0,$
\begin{align*}
\int_{\mathbb{R}^{3}}w_{0}\varphi dx+\int^{\infty}_{0}\int_{\mathbb{R}^{3}}\left(a(\partial_{t}\varphi+\Delta\varphi+u^{c,\gamma}\cdot\nabla\varphi)-(a\cdot\nabla)u^{c,\gamma}\varphi\right)dxdt=0.
\end{align*}

With regard to the linear part $a$, we have
\begin{lemma}\label{Lem01}
For $p\geq2$ and $w_{0}\in L^{p}_{\sigma}(\mathbb{R}^{3})$, there exists a small constant $\delta_{p}>0$ such that if $(c,\gamma)\in M\cap\{|(c,\gamma)|\leq\delta_{p}\}$, then equations \eqref{LINE001} has a unique global-in-time solution $a\in C_{t}L_{x}^{p}\cap L_{t}^{\frac{p}{\alpha}}L_{x}^{\frac{3p}{3-2\alpha}}$ for any $0<\alpha\leq1$. Moreover, this solution satisfies
\begin{align}\label{ESTI001}
\|a(\cdot,s)\|_{L^{p}}\leq\|a(\cdot,t)\|_{L^{p}},\quad\text{for any $0\leq t\leq s<\infty$},
\end{align}
and, for any $0<\alpha\leq1,$
\begin{align*}
\|a\|_{C_{t}L_{x}^{p}\cap L_{t}^{\frac{p}{\alpha}}L_{x}^{\frac{3p}{3-2\alpha}}}+\big\|\nabla|a|^{\frac{p}{2}}\big\|^{\frac{2}{p}}_{L_{t}^{2}L_{x}^{2}}\leq C(p,\alpha,c,\gamma)\|w_{0}\|_{L^{p}}.
\end{align*}

\end{lemma}
\begin{remark}
We will prove Lemmas \ref{Lem01} and \ref{Lem02} by combining the energy estimates and the method of induction.

\end{remark}

\begin{proof}
To begin with, we show the existence of $a$ by using the classical Picard iteration scheme. Set $a_{0}=0$ and construct an iterative sequence $\{a_{k}\}$ satisfying the following equations: for $k\geq1$,
\begin{align}\label{QTN001}
\begin{cases}
\partial_{t}a_{k}-\Delta a_{k}=-\mathrm{div}(u^{c,\gamma}\otimes a_{k-1}+a_{k-1}\otimes u^{c,\gamma})-\nabla\pi_{k-1},\\
\mathrm{div}a_{k}=0,\\
\pi_{k-1}=(-\Delta)^{-1}\partial_{i}\partial_{j}(u^{c,\gamma}\otimes a_{k-1}+a_{k-1}\otimes u^{c,\gamma}),\\
a_{k}(x,0)=w_{0}.
\end{cases}
\end{align}

{\bf Step 1.} Claim that for $k\geq1$, $a_{k}\in L^{\infty}([0,\infty);L^{p}(\mathbb{R}^{3}))\cap L^{\frac{p}{\alpha}}([0,\infty);L^{\frac{3p}{3-2\alpha}}(\mathbb{R}^{3}))$ with any $0<\alpha\leq1$, and $\nabla|a_{k}|^{\frac{p}{2}}\in L^{2}([0,\infty);L^{2}(\mathbb{R}^{3}))$. For $k=1,$ assume without loss of generality that $a_{1}$ is smooth. Then multiplying equation \eqref{QTN001} by $|a_{1}|^{p-2}a_{1}$ with $p\geq2$, it follows from integration by parts that
\begin{align*}
\begin{cases}
\frac{1}{2}\frac{d}{dt}\|a_{1}\|_{L^{2}}^{2}+\|\nabla a_{1}\|_{L^{2}}^{2}=0,&p=2,\\
\frac{1}{p}\frac{d}{dt}\|a_{1}\|_{L^{p}}^{p}+\frac{4(p-2)}{p^{2}}\|\nabla|a_{1}|^{\frac{p}{2}}\|_{L^{2}}^{2}+\||\nabla a_{1}||a_{1}|^{\frac{p-2}{2}}\|_{L^{2}}^{2}=0,&p>2,
\end{cases}
\end{align*}
which implies that
\begin{align*}
\|a_{1}\|_{L_{t}^{\infty}L_{x}^{p}}+\big\|\nabla|a_{1}|^{\frac{p}{2}}\big\|^{\frac{2}{p}}_{L_{t}^{2}L_{x}^{2}}\leq \|w_{0}\|_{L^{p}}
\begin{cases}
1+\frac{1}{\sqrt{2}},&p=2,\\
1+\sqrt[p]{\frac{p-2}{4p}},&p>2.
\end{cases}
\end{align*}
Applying the Gagliardo-Nirenberg interpolation inequality, we obtain that for $0<\alpha\leq1,$
\begin{align*}
\|a_{1}\|_{L_{t}^{\frac{p}{\alpha}}L_{x}^{\frac{3p}{3-2\alpha}}}\leq C(p,\alpha)\|a_{1}\|_{L_{t}^{\infty}L_{x}^{p}}^{1-\alpha}\big\|\nabla|a_{1}|^{\frac{p}{2}}\big\|^{\frac{2\alpha}{p}}_{L_{t}^{2}L_{x}^{2}}.
\end{align*}
Then we deduce that for $0<\alpha\leq1,$
\begin{align}\label{ZZW001}
\|a_{1}\|_{L^{\infty}_{t}L_{x}^{p}\cap L_{t}^{\frac{p}{\alpha}}L_{x}^{\frac{3p}{3-2\alpha}}}+\big\|\nabla|a_{1}|^{\frac{p}{2}}\big\|^{\frac{2}{p}}_{L_{t}^{2}L_{x}^{2}}\leq C(p,\alpha)\|w_{0}\|_{L^{p}}.
\end{align}

Assume that there exist some constant $C=C(p,\alpha,c,\gamma)>0$ such that \eqref{ZZW001} holds with $a_{1}$ replaced by $a_{k-1}$ for $k\geq3$, that is, for $0<\alpha\leq1$,
\begin{align}\label{ASUM001}
\|a_{k-1}\|_{L^{\infty}_{t}L_{x}^{p}\cap L_{t}^{\frac{p}{\alpha}}L_{x}^{\frac{3p}{3-2\alpha}}}+\big\|\nabla|a_{k-1}|^{\frac{p}{2}}\big\|^{\frac{2}{p}}_{L_{t}^{2}L_{x}^{2}}\leq C\|w_{0}\|_{L^{p}}.
\end{align}
Then we next prove that \eqref{ASUM001} also holds for $a_{k}$. Suppose $a_{k}$ is smooth. Multiplying \eqref{QTN001} by $|a_{k}|^{p-2}a_{k}$ and integrating by parts, we have
\begin{align}\label{ZZW002}
&\frac{1}{p}\frac{d}{dt}\|a_{k}\|_{L^{p}}^{p}+\frac{4(p-2)}{p^{2}}\|\nabla|a_{k}|^{\frac{p}{2}}\|_{L^{2}}^{2}+\||\nabla a_{k}||a_{k}|^{\frac{p-2}{2}}\|_{L^{2}}^{2}\notag\\
&=\int_{\mathbb{R}^{3}}(u^{c,\gamma}\otimes a_{k-1}+a_{k-1}\otimes u^{c,\gamma})\cdot\nabla(|a_{k}|^{p-2}a_{k})+\int_{\mathbb{R}^{3}}\pi_{k-1}\mathrm{div}(|a_{k}|^{p-2}a_{k}).
\end{align}
Applying Corollary \ref{MZLEM001} and Lemma \ref{CORO001}, it follows from H\"{o}lder's inequality and Young's inequality that
\begin{align}\label{ZZW003}
&\left|\int_{\mathbb{R}^{3}}(u^{c,\gamma}\otimes a_{k-1}+a_{k-1}\otimes u^{c,\gamma})\cdot\nabla(|a_{k}|^{p-2}a_{k})\right|\notag\\
&\leq C\int_{\mathbb{R}^{3}}|\nabla|a_{k}|^{\frac{p}{2}}||a_{k}|^{\frac{p}{2}-1}|a_{k-1}||u^{c,\gamma}|\notag\\
&= C\int_{\Omega_{e^{-1}}}|\nabla|a_{k}|^{\frac{p}{2}}|\frac{|a_{k}|^{\frac{p-2}{2}}}{(|x'||x|)^{\frac{p-2}{2p}}}\frac{|a_{k-1}|}{(|x'||x|)^{\frac{1}{p}}}(|x'||x|)^{\frac{1}{2}}|u^{c,\gamma}|\notag\\
&\quad+C\int_{\mathbb{R}^{3}\setminus\Omega_{e^{-1}}}|\nabla|a_{k}|^{\frac{p}{2}}|\frac{|a_{k}|^{\frac{p-2}{2}}}{|x|^{\frac{p-2}{p}}}\frac{|a_{k-1}|}{|x|^{\frac{2}{p}}}|x||u^{c,\gamma}|\notag\\
&\leq CK(c,\gamma)\|\nabla|a_{k}|^{\frac{p}{2}}\|_{L^{2}}\left(\int_{\mathbb{R}^{3}}\frac{|a_{k}|^{p}}{|x'||x|}\right)^{\frac{p-2}{2p}}\left(\int_{\mathbb{R}^{3}}\frac{|a_{k-1}|^{p}}{|x'||x|}\right)^{\frac{1}{p}}\notag\\
&\quad+CK(c,\gamma)\|\nabla|a_{k}|^{\frac{p}{2}}\|_{L^{2}}\left(\int_{\mathbb{R}^{3}}\frac{|a_{k}|^{p}}{|x|^{2}}\right)^{\frac{p-2}{2p}}\left(\int_{\mathbb{R}^{3}}\frac{|a_{k-1}|^{p}}{|x|^{2}}\right)^{\frac{1}{p}}\notag\\
&\leq C_{0}K(c,\gamma)\big\|\nabla|a_{k}|^{\frac{p}{2}}\big\|^{\frac{2(p-1)}{p}}_{L^{2}}\big\|\nabla|a_{k-1}|^{\frac{p}{2}}\big\|^{\frac{2}{p}}_{L^{2}}\notag\\
&\leq\frac{C_{0}(p-1)K(c,\gamma)}{p}\big\|\nabla|a_{k}|^{\frac{p}{2}}\big\|^{2}_{L^{2}}+\frac{C_{0}K(c,\gamma)}{p}\big\|\nabla|a_{k-1}|^{\frac{p}{2}}\big\|^{2}_{L^{2}},
\end{align}
where $\Omega_{e^{-1}}$ is a cone defined by \eqref{CONE999}.

Using Lemma \ref{MZLEM002}, we deduce that $(|x'||x|)^{\frac{p-2}{2}}$ and $|x|^{p-2}$ are all $A_{p}$-weights. Since $(-\Delta)^{-1}\partial_{i}\partial_{j}$ is a Calder\'{o}n-Zygmund operator, it then follows from the boundedness of the Riesz transforms on weighted $L^{p}$ spaces (see Theorem 9.4.6 in \cite{G2009}), H\"{o}lder's inequality, Young's inequality, Corollary \ref{MZLEM001} and Lemma \ref{CORO001} that

$(i)$ if $p=2$, then
\begin{align}\label{ZZW005}
\int_{\mathbb{R}^{3}}\pi_{k-1}\mathrm{div}a_{k}=0;
\end{align}

$(ii)$ if $p>2$, then
\begin{align}\label{ZZW006}
&\int_{\mathbb{R}^{3}}\pi_{k-1}\mathrm{div}(|a_{k}|^{p-2}a_{k})\leq C\int_{\mathbb{R}^{3}}|\pi_{k-1}||\nabla|a_{k}|^{\frac{p}{2}}||a_{k}|^{\frac{p}{2}-1}\notag\\
&\leq C\big\|(|x'||x|)^{\frac{p-2}{2p}}|a_{k-1}\otimes u^{c,\gamma}|\big\|_{L^{p}(\Omega_{e^{-1}})}\|\nabla|a_{k}|^{\frac{p}{2}}\|_{L^{2}}\bigg\|\frac{|a_{k}|^{\frac{p-2}{2}}}{(|x'||x|)^{\frac{p-2}{2p}}}\bigg\|_{L^{\frac{2p}{p-2}}}\notag\\
&\quad+C\big\||x|^{\frac{p-2}{p}}|a_{k-1}\otimes u^{c,\gamma}|\big\|_{L^{p}(\mathbb{R}^{3}\setminus\Omega_{e^{-1}})}\|\nabla|a_{k}|^{\frac{p}{2}}\|_{L^{2}}\bigg\|\frac{|a_{k}|^{\frac{p-2}{2}}}{|x|^{\frac{p-2}{p}}}\bigg\|_{L^{\frac{2p}{p-2}}}\notag\\
&\leq C_{0}K(c,\gamma)\big\|\nabla|a_{k-1}|^{\frac{p}{2}}\big\|^{\frac{2}{p}}_{L^{p}}\big\|\nabla|a_{k}|^{\frac{p}{2}}\big\|_{L^{2}}^{\frac{2(p-1)}{p}}\notag\\
&\leq\frac{C_{0}(p-1)K(c,\gamma)}{p}\big\|\nabla|a_{k}|^{\frac{p}{2}}\big\|_{L^{2}}^{2}+\frac{C_{0}K(c,\gamma)}{p}\|\nabla|a_{k-1}|^{\frac{p}{2}}\big\|^{2}_{L^{p}}.
\end{align}
Then substituting \eqref{ZZW003}--\eqref{ZZW006} into \eqref{ZZW002}, we obtain that there exists a small constant $\delta_{p}>0$ such that if $|(c,\gamma)|<\delta_{p}$, then
\begin{align}\label{ZZW007}
C_{0}K(c,\gamma)<
\begin{cases}
1,&p=2,\\
\min\Big\{\frac{1}{2},\frac{2(p-2)}{p^{2}}\Big\},&p>2,
\end{cases}
\end{align}
and thus,
\begin{align}\label{ZZW008}
&\frac{d}{dt}\|a_{k}(t)\|_{L^{p}}^{p}+\|\nabla|a_{k}|^{\frac{p}{2}}\|^{2}_{L^{2}}\leq\tau_{p} \|\nabla|a_{k-1}|^{\frac{p}{2}}\|^{2}_{L^{2}},
\end{align}
where
\begin{align}\label{MCQ009}
\tau_{p}:=
\begin{cases}
C_{0}K(c,\gamma),&p=2,\\
\frac{C_{0}K(c,\gamma)}{\min\{\frac{1}{2},\frac{2(p-2)}{p}-C_{0}(p-1)K(c,\gamma)\}},&p>2.
\end{cases}
\end{align}
Therefore, integrating \eqref{ZZW008} from $0$ to $T$ with any $T>0$, it follows from \eqref{ASUM001} that
\begin{align*}
\|a_{k}(T)\|_{L^{p}}^{p}+\|\nabla|a_{k}|^{\frac{p}{2}}\|_{L_{T}^{2}L_{x}^{2}}^{2}\leq&\tau_{p}\|\nabla|a_{k-1}|^{\frac{p}{2}}\|_{L_{T}^{2}L_{x}^{2}}^{2}+\|w_{0}\|_{L^{p}}^{p}\leq(1+C\tau_{p})\|w_{0}\|_{L^{p}}^{p},\notag\\
\end{align*}
which, together with the interpolation inequality, shows that for $0<\alpha\leq1$,
\begin{align}\label{ZWZ90}
\|a_{k}\|_{L^{\infty}_{t}L_{x}^{p}\cap L_{t}^{\frac{p}{\alpha}}L_{x}^{\frac{3p}{3-2\alpha}}}+\big\|\nabla|a_{k}|^{\frac{p}{2}}\big\|^{\frac{2}{p}}_{L_{t}^{2}L_{x}^{2}}\leq C(p,\alpha,c,\gamma)\|w_{0}\|_{L^{p}}.
\end{align}

{\bf Step 2.} We now prove that $\{a_{k}\}$ is a Cauchy sequence in $L_{t}^{\infty}L_{x}^{p}\cap L_{t}^{\frac{p}{\alpha}}L_{x}^{\frac{3p}{3-2\alpha}}$ for any $0<\alpha\leq1$. For $k\geq1$, $a_{k+1}-a_{k}$ satisfies
\begin{align}\label{CAU001}
\begin{cases}
\partial_{t}(a_{k+1}-a_{k})-\Delta (a_{k+1}-a_{k})\\
=-\mathrm{div}(u^{c,\gamma}\otimes (a_{k}-a_{k-1})+(a_{k}-a_{k-1})\otimes u^{c,\gamma})-\nabla(\pi_{k}-\pi_{k-1}),\\
\mathrm{div}(a_{k+1}-a_{k})=0,\\
\pi_{k}-\pi_{k-1}=(-\Delta)^{-1}\partial_{i}\partial_{j}(u^{c,\gamma}\otimes(a_{k}-a_{k-1})+(a_{k}-a_{k-1})\otimes u^{c,\gamma}),\\
(a_{k+1}-a_{k})(x,0)=0.
\end{cases}
\end{align}
Without loss of generality, suppose that $\{a_{k}\}$ is a smooth sequence. Multiplying the above equation by $|a_{k+1}-a_{k}|^{p-2}(a_{k+1}-a_{k})$ and integrating by parts, we have
\begin{align*}
&\frac{1}{p}\frac{d}{dt}\|a_{k+1}-a_{k}\|_{L^{p}}^{p}+\frac{4(p-2)}{p^{2}}\|\nabla|a_{k+1}-a_{k}|^{\frac{p}{2}}\|_{L^{2}}^{2}+\||\nabla(a_{k+1}-a_{k})||a_{k+1}-a_{k}|^{\frac{p-2}{2}}\|_{L^{2}}^{2}\notag\\
&=\int_{\mathbb{R}^{3}}(u^{c,\gamma}\otimes (a_{k}-a_{k-1})+(a_{k}-a_{k-1})\otimes u^{c,\gamma})\cdot\nabla(|a_{k+1}-a_{k}|^{p-2}(a_{k+1}-a_{k}))\notag\\
&\quad+\int_{\mathbb{R}^{3}}(\pi_{k}-\pi_{k-1})\mathrm{div}(|a_{k+1}-a_{k}|^{p-2}(a_{k+1}-a_{k})).
\end{align*}
By the same argument as in \eqref{ZZW008}, we have
\begin{align*}
&\frac{d}{dt}\|(a_{k+1}-a_{k})(t)\|_{L^{p}}^{p}+\|\nabla|a_{k+1}-a_{k}|^{\frac{p}{2}}\|^{2}_{L^{2}}\leq\tau_{p} \|\nabla|a_{k}-a_{k-1}|^{\frac{p}{2}}\|^{2}_{L^{2}},
\end{align*}
where $\tau_{p}$ is given by \eqref{MCQ009}. In light of \eqref{ZZW007}, we know that $\tau_{p}<1$. Therefore, we derive
\begin{align*}
&\|a_{k+1}-a_{k}\|_{L_{t}^{\infty}L_{x}^{p}}+\|\nabla|a_{k+1}-a_{k}|^{\frac{p}{2}}\|^{\frac{2}{p}}_{L^{2}_{t}L^{2}_{x}}\notag\\
&\leq\sqrt[p]{\tau_{p}}\big( \|a_{k}-a_{k-1}\|_{L_{t}^{\infty}L_{x}^{p}}+\|\nabla|a_{k}-a_{k-1}|^{\frac{p}{2}}\|^{\frac{2}{p}}_{L^{2}_{t}L^{2}_{x}}\big),
\end{align*}
which implies that $\{a_{k}\}$ is a Cauchy sequence under this norm. Then there exists a limit $a\in L_{t}^{\infty}L_{x}^{p}$ satisfying $\nabla|a|^{\frac{p}{2}}\in L_{t}^{2}L_{x}^{2}$ such that
\begin{align*}
\lim\limits_{k\rightarrow\infty}\|a_{k}-a\|_{L_{t}^{\infty}L_{x}^{p}}+\|\nabla|a_{k}-a|^{\frac{p}{2}}\|^{\frac{2}{p}}_{L^{2}_{t}L^{2}_{x}}=0.
\end{align*}
This, in combination with the Gagliardo-Nirenberg interpolation inequality, leads to that for $0<\alpha\leq1,$
\begin{align*}
\lim\limits_{k\rightarrow\infty}\|a_{k}-a\|_{L_{t}^{\infty}L_{x}^{p}\cap L_{t}^{\frac{p}{\alpha}}L_{x}^{\frac{3p}{3-2\alpha}}}+\|\nabla|a_{k}-a|^{\frac{p}{2}}\|^{\frac{2}{p}}_{L^{2}_{t}L^{2}_{x}}=0.
\end{align*}
Then sending $k\rightarrow\infty$ in \eqref{ZWZ90}, we deduce
\begin{align*}
\|a\|_{L^{\infty}_{t}L_{x}^{p}\cap L_{t}^{\frac{p}{\alpha}}L_{x}^{\frac{3p}{3-2\alpha}}}+\big\|\nabla|a|^{\frac{p}{2}}\big\|^{\frac{2}{p}}_{L_{t}^{2}L_{x}^{2}}\leq C(p,\alpha,c,\gamma)\|w_{0}\|_{L^{p}}.
\end{align*}

{\bf Step 3.} It remains to show that $a\in C([0,\infty);L_{x}^{p})$. Due to the translational invariance in time, it suffices to demonstrate the continuity of $a$ near $t=0.$ For this purpose, it only needs to show that for any sequence $t_{i}\rightarrow0,$ as $i\rightarrow\infty,$ there holds
\begin{align}\label{SEQ001}
\lim\limits_{i\rightarrow\infty}\|a(\cdot,t_{i})-w_{0}\|_{L^{p}}=0.
\end{align}
From \eqref{ESTI001}, we know
\begin{align}\label{BOUN001}
\|a(\cdot,t_{i})\|_{L^{p}}\leq\|w_{0}\|_{L^{p}},\quad\text{for each $i\geq1$}.
\end{align}
According to the weak compactness of reflexive Banach space, we derive that there exists a subsequence $\{t_{i_{j}}\}$ such that
\begin{align*}
a(\cdot,t_{i_{j}})\rightharpoonup w_{0}\quad\text{weakly in $L^{p}$, \;as $j\rightarrow\infty$}.
\end{align*}
This yields that
\begin{align*}
\|w_{0}\|_{L^{p}}\leq\liminf\limits_{j\rightarrow\infty}\|a(\cdot,t_{i_{j}})\|_{L^{p}}.
\end{align*}
On the other hand, using the boundedness of $\{\|a(\cdot,t_{i})\|_{L^{p}}\}$ in \eqref{BOUN001}, we have
\begin{align*}
\limsup\limits_{j\rightarrow\infty}\|a(\cdot,t_{i_{j}})\|_{L^{p}}\leq\|w_{0}\|_{L^{p}}.
\end{align*}
Then we obtain
\begin{align*}
\lim\limits_{j\rightarrow\infty}\|a(\cdot,t_{i_{j}})\|_{L^{p}}=\|w_{0}\|_{L^{p}}.
\end{align*}
Since $\|a(\cdot,t_{i})\|_{L^{p}}$ is decreasing as $i$ increases, then we further deduce that \eqref{SEQ001} also holds. The proof is complete.

\end{proof}

For any $w_{1},w_{2}\in L^{\frac{4p}{3}}([0,T];L^{2p}_{x}(\mathbb{R}^{3}))$ with $T>0$, we prepare to estimate the nonlinear part $N(w_{1},w_{2})$ as follows:
\begin{align*}
N(w_{1},w_{2})=-\int^{t}_{0}e^{-(t-s)\mathcal{L}}\mathbb{P}\mathrm{div}(w_{1}\otimes w_{2})ds.
\end{align*}
For simplicity, write $z:=N(w_{1},w_{2})$. Then $z$ satisfies the following equation
\begin{align}\label{NONLINE002}
\begin{cases}
\partial_{t}z-\Delta z+(z\cdot\nabla)u^{c,\gamma}+(u^{c,\gamma}\cdot\nabla)z+\nabla\pi_{2}=-\mathrm{div}(w_{1}\otimes w_{2}),\\
\mathrm{div}z=0,\\
z(x,0)=0.
\end{cases}
\end{align}
Namely, for any $\varphi\in C^{\infty}_{c}([0,T)\times\mathbb{R}^{3})$ with $\mathrm{div}\varphi=0,$
\begin{align*}
\int^{T}_{0}\int_{\mathbb{R}^{3}}\left(z(\partial_{t}\varphi+\Delta\varphi+u^{c,\gamma}\cdot\nabla\varphi)+(w_{1}\otimes w_{2})\cdot\nabla\varphi-(z\cdot\nabla)u^{c,\gamma}\varphi\right)dxdt=0.
\end{align*}

\begin{lemma}\label{Lem02}
Let $p\geq3$ and $T>0$. For any $w_{1},w_{2}\in L^{\frac{4p}{3}}([0,T];L^{2p}(\mathbb{R}^{3}))$, there exists a small constant $\delta_{p}>0$ such that if $(c,\gamma)\in M\cap\{|(c,\gamma)|\leq\delta_{p}\}$, then equations \eqref{NONLINE002} has a unique solution $z\in C_{T}L_{x}^{p}\cap L_{T}^{\frac{p}{\alpha}}L_{x}^{\frac{3p}{3-2\alpha}}$ for any $0<\alpha\leq1$. Moreover, this solution satisfies that for any $0<\alpha\leq1,$
\begin{align*}
\|z\|_{C_{T}L_{x}^{p}\cap L_{T}^{\frac{p}{\alpha}}L_{x}^{\frac{3p}{3-2\alpha}}}+\big\|\nabla|z|^{\frac{p}{2}}\big\|^{\frac{2}{p}}_{L_{T}^{2}L_{x}^{2}}\leq C(p,\alpha,c,\gamma)
\begin{cases}
\|w_{1}\|_{L_{T}^{\frac{4p}{2p-3}}L_{x}^{2p}}\|w_{2}\|_{L_{T}^{\frac{4p}{3}}L_{x}^{2p}},\\
T^{\frac{p-3}{2p}}\|w_{1}\|_{L_{T}^{\frac{4p}{3}}L_{x}^{2p}}\|w_{2}\|_{L_{T}^{\frac{4p}{3}}L_{x}^{2p}}.
\end{cases}
\end{align*}

\end{lemma}

\begin{proof}
Similar to Lemma \ref{Lem01}, we will use the classical Picard iteration scheme to solve the existence of $z$. Pick $z_{0}=0$. For any given $w_{1},w_{2}\in L^{\frac{4p}{3}}([0,T];L^{2p}(\mathbb{R}^{3}))$, we construct an iterative sequence $\{z_{k}\}$ verifying the following equations: for $k\geq1$,
\begin{align}\label{QWMA01}
\begin{cases}
\partial_{t}z_{k}-\Delta z_{k}=-\mathrm{div}(u^{c,\gamma}\otimes z_{k-1}+z_{k-1}\otimes u^{c,\gamma}+w_{1}\otimes w_{2})-\nabla\pi_{k-1},\\
\mathrm{div}z_{k}=0,\\
\pi_{k-1}=(-\Delta)^{-1}\partial_{i}\partial_{j}(u^{c,\gamma}\otimes z_{k-1}+z_{k-1}\otimes u^{c,\gamma}+w_{1}\otimes w_{2}),\\
z_{k}(x,0)=0.
\end{cases}
\end{align}

{\bf Step 1.} Claim that for $k\geq1$, $z_{k}\in L^{\infty}([0,T];L^{p}(\mathbb{R}^{3}))\cap L^{\frac{p}{\alpha}}([0,T];L^{\frac{3p}{3-2\alpha}}(\mathbb{R}^{3}))$ with any $0<\alpha\leq1$, and $\nabla|z_{k}|^{\frac{p}{2}}\in L^{2}([0,T];L^{2}(\mathbb{R}^{3}))$. For $k=1,$ suppose that $z_{1}$ is smooth. Multiplying equation \eqref{QWMA01} by $|z_{1}|^{p-2}z_{1}$ with $p\geq3$, we deduce from integration by parts that
\begin{align*}
&\frac{1}{p}\frac{d}{dt}\|z_{1}\|_{L^{p}}^{p}+\frac{4(p-2)}{p^{2}}\|\nabla|z_{1}|^{\frac{p}{2}}\|_{L^{2}}^{2}+\||\nabla z_{1}||a_{1}|^{\frac{p-2}{2}}\|_{L^{2}}^{2}\notag\\
&=\int_{\mathbb{R}^{3}}(w_{1}\otimes w_{2})\cdot\nabla(|z_{1}|^{p-2}z_{1})+\int_{\mathbb{R}^{3}}(-\Delta)^{-1}\partial_{i}\partial_{j}(w_{1}\otimes w_{2})\mathrm{div}(|z_{1}|^{p-2}z_{1}):=I_{1}+I_{2}.
\end{align*}
As for the first term $I_{1}$, from H\"{o}lder's inequality and Young's inequality, we deduce
\begin{align}\label{ZEW001}
|I_{1}|\leq&C\int_{\mathbb{R}^{3}}|\nabla|z_{1}|^{\frac{p}{2}}||z_{1}|^{\frac{p}{2}-1}|w_{1}\otimes w_{2}|\notag\\
\leq&C\|\nabla|z_{1}|^{\frac{p}{2}}\|_{L^{2}}\|z_{1}\|_{L^{p}}^{\frac{p-2}{2}}\|w_{1}\otimes w_{2}\|_{L^{p}}\notag\\
\leq&\frac{4(p-2)}{3p^{2}}\|\nabla|z_{1}|^{\frac{p}{2}}\|_{L^{2}}^{2}+C(p)\|z_{1}\|_{L^{p}}^{p-2}\|w_{1}\otimes w_{2}\|_{L^{p}}^{2}.
\end{align}
Observe from the property of scalar Riesz operator (see Theorem 1.1 in \cite{IM1996}) that
\begin{align}\label{QMYU001}
\|(-\Delta)^{-1}\partial_{i}\partial_{j}(w_{1}\otimes w_{1})\|_{L^{p}}\leq&
\overline{C}_{0}H_{p}\|w_{1}\otimes w_{2}\|_{L^{p}},
\end{align}
where $\overline{C}_{0}$ is a universal constant independent of $p$ and $H_{p}=\cot(\pi/2p)$. Then it follows from H\"{o}lder's inequality and Young's inequality again that
\begin{align}\label{ZEW002}
|I_{2}|\leq&\frac{2(p-2)}{p}\int_{\mathbb{R}^{3}}|(-\Delta)^{-1}\partial_{i}\partial_{j}(w_{1}\otimes w_{2})||\nabla|z_{1}|^{\frac{p}{2}}||z_{1}|^{\frac{p-2}{2}}\notag\\
\leq&\frac{2(p-2)}{p}\overline{C}_{0}H_{p}\|\nabla|z_{1}|^{\frac{p}{2}}\|_{L^{2}}\|z_{1}\|^{\frac{p-2}{2}}_{L^{p}}\|w_{1}\otimes w_{2}\|_{L^{p}}\notag\\
\leq&\frac{4(p-2)}{3p^{2}}\|\nabla|z_{1}|^{\frac{p}{2}}\|_{L^{2}}^{2}+C(p)\|z_{1}\|_{L^{p}}^{p-2}\|w_{1}\otimes w_{2}\|_{L^{p}}^{2},
\end{align}
Combining the above results, we deduce that
\begin{align}\label{ZMA001}
\frac{d}{dt}\|z_{1}\|_{L^{p}}^{p}+\frac{4(p-2)}{3p}\|\nabla|z_{1}|^{\frac{p}{2}}\|_{L^{2}}^{2}\leq C(p)\|z_{1}\|_{L^{p}}^{p-2}\|w_{1}\otimes w_{2}\|_{L^{p}}^{2}.
\end{align}
Integrating it from $0$ to $s$ with any $0<s\leq T$, we obtain from Young's inequality that
\begin{align*}
\|z_{1}(s)\|_{L^{p}}^{p}\leq& C(p)\|z_{1}\|_{L^{\infty}_{T}L_{x}^{p}}\|w_{1}\otimes w_{2}\|^{2}_{L_{t}^{2}L_{x}^{p}}\notag\\
\leq&\frac{1}{2}\|z_{1}\|_{L^{\infty}_{T}L_{x}^{p}}^{p}+C(p)\|w_{1}\otimes w_{2}\|^{p}_{L_{t}^{2}L_{x}^{p}},
\end{align*}
which yields that
\begin{align}\label{ZQMAM001}
\|z_{1}\|_{L^{\infty}_{T}L_{x}^{p}}\leq C(p)\|w_{1}\otimes w_{2}\|_{L_{t}^{2}L_{x}^{p}}.
\end{align}
Substituting \eqref{ZQMAM001} into \eqref{ZMA001} and integrating it from $0$ to $T$, we obtain that
\begin{align*}
\|\nabla|z_{1}|^{\frac{p}{2}}\|_{L_{T}^{2}L^{2}}^{\frac{2}{p}}\leq C(p)\|z_{1}\|_{L_{T}^{\infty}L_{x}^{p}}^{\frac{p-2}{p}}\|w_{1}\otimes w_{2}\|_{L_{T}^{2}L_{x}^{p}}^{\frac{2}{p}}\leq C(p)\|w_{1}\otimes w_{2}\|_{L_{T}^{2}L_{x}^{p}}.
\end{align*}
This, together with \eqref{ZQMAM001} and the interpolation inequality, gives that for $0<\alpha\leq1$,
\begin{align}\label{ZWMK001}
\|z_{1}\|_{L^{\infty}_{T}L_{x}^{p}\cap L_{T}^{\frac{p}{\alpha}}L_{x}^{\frac{3p}{3-2\alpha}}}+\big\|\nabla|z_{1}|^{\frac{p}{2}}\big\|^{\frac{2}{p}}_{L_{T}^{2}L_{x}^{2}}\leq& C(p,\alpha)\|w_{1}\otimes w_{2}\|_{L_{T}^{2}L_{x}^{p}}.
\end{align}

Suppose that there exist some constant $C=C(p,\alpha,c,\gamma)>0$ such that \eqref{ZWMK001} holds with $z_{1}$ replaced by $z_{k-1}$ for $k\geq3$, namely, for $0<\alpha\leq1$,
\begin{align}\label{DQ001}
\|z_{k-1}\|_{L^{\infty}_{T}L_{x}^{p}\cap L_{T}^{\frac{p}{\alpha}}L_{x}^{\frac{3p}{3-2\alpha}}}+\big\|\nabla|z_{k-1}|^{\frac{p}{2}}\big\|^{\frac{2}{p}}_{L_{T}^{2}L_{x}^{2}}\leq C\|w_{1}\otimes w_{2}\|_{L_{T}^{2}L_{x}^{p}}.
\end{align}
Then we need to demonstrate that \eqref{DQ001} also holds for $z_{k}$. Assume that $z_{k}$ is smooth. Multiplying \eqref{QWMA01} by $|z_{k}|^{p-2}z_{k}$, we have from integration by parts that
\begin{align}\label{ZEMAA66}
&\frac{1}{p}\frac{d}{dt}\|z_{k}\|_{L^{p}}^{p}+\frac{4(p-2)}{p^{2}}\|\nabla|z_{k}|^{\frac{p}{2}}\|_{L^{2}}^{2}+\||\nabla z_{k}||a_{1}|^{\frac{p-2}{2}}\|_{L^{2}}^{2}\notag\\
&=\int_{\mathbb{R}^{3}}(w_{1}\otimes w_{2})\cdot\nabla(|z_{k}|^{p-2}z_{k})+\int_{\mathbb{R}^{3}}(-\Delta)^{-1}\partial_{i}\partial_{j}(w_{1}\otimes w_{2})\mathrm{div}(|z_{k}|^{p-2}z_{k})\notag\\
&\quad+\int_{\mathbb{R}^{3}}(u^{c,\gamma}\otimes z_{k-1}+z_{k-1}\otimes u^{c,\gamma})\cdot\nabla(|z_{k}|^{p-2}z_{k})\notag\\
&\quad+\int_{\mathbb{R}^{3}}(-\Delta)^{-1}\partial_{i}\partial_{j}(u^{c,\gamma}\otimes z_{k-1}+z_{k-1}\otimes u^{c,\gamma})\mathrm{div}(|z_{k}|^{p-2}z_{k})\notag\\
&:=I_{1}+I_{2}+I_{3}+I_{4}.
\end{align}
In exactly the same way to \eqref{ZEW001} and \eqref{ZEW002}, we have
\begin{align}\label{ZEMA01}
|I_{1}+I_{2}|\leq&\frac{8(p-2)}{3p^{2}}\|\nabla|z_{k}|^{\frac{p}{2}}\|_{L^{2}}^{2}+C(p)\|z_{k}\|_{L^{p}}^{p-2}\|w_{1}\otimes w_{2}\|_{L^{p}}^{2}.
\end{align}
With regard to the terms $I_{3}$ and $I_{4}$, by the same arguments as in \eqref{ZZW003}--\eqref{ZZW006}, it follows from \eqref{DQ001} that
\begin{align}\label{ZEMA09}
|I_{3}+I_{4}|\leq&\frac{2C_{0}(p-1)K(c,\gamma)}{p}\big\|\nabla|z_{k}|^{\frac{p}{2}}\big\|^{2}_{L^{2}}+\frac{2C_{0}K(c,\gamma)}{p}\big\|\nabla|z_{k-1}|^{\frac{p}{2}}\big\|^{2}_{L^{2}}.
\end{align}
Pick a small constant $\bar{\delta}_{p}>0$ such that if $|(c,\gamma)|<\bar{\delta}_{p}$, there holds
\begin{align*}
C_{0}K(c,\gamma)<\min\Big\{\frac{1}{2},\frac{2(p-2)}{3p^{2}}\Big\}.
\end{align*}
Then inserting \eqref{ZEMA01} and \eqref{ZEMA09} into \eqref{ZEMAA66}, we deduce
\begin{align}\label{ZQMP001}
&\frac{d}{dt}\|z_{k}(t)\|_{L^{p}}^{p}+\|\nabla|z_{k}|^{\frac{p}{2}}\|^{2}_{L^{2}}\leq\bar{\tau}_{p} \|\nabla|z_{k-1}|^{\frac{p}{2}}\|^{2}_{L^{2}}+C\|z_{k}\|_{L^{p}}^{p-2}\|w_{1}\otimes w_{2}\|_{L^{p}}^{2},
\end{align}
where
\begin{align}\label{ZQMP002}
\bar{\tau}_{p}:=
\frac{C_{0}K(c,\gamma)}{\min\{\frac{1}{2},\frac{2(p-2)}{3p}-C_{0}(p-1)K(c,\gamma)\}}<1.
\end{align}
Consequently, integrating \eqref{ZQMP001} from $0$ to $t$ with $0<t\leq T$, we have from \eqref{DQ001} and Young's inequality that
\begin{align*}
&\|z_{k}(t)\|_{L^{p}}^{p}+\|\nabla|z_{k}|^{\frac{p}{2}}\|_{L_{t}^{2}L_{x}^{2}}^{2}\notag\\
&\leq\bar{\tau}_{p}\|\nabla|z_{k-1}|^{\frac{p}{2}}\|_{L_{T}^{2}L_{x}^{2}}^{2}+C\|z_{k}\|_{L^{p}}^{p-2}\|w_{1}\otimes w_{2}\|_{L^{p}}^{2}\notag\\
&\leq\frac{1}{2}\|z_{k}\|_{L^{\infty}_{T}L_{x}^{p}}^{p}+C(1+\bar{\tau}_{p})\|w_{1}\otimes w_{2}\|^{p}_{L_{t}^{2}L_{x}^{p}},
\end{align*}
which, in combination with the interpolation inequality, leads to that for $0<\alpha\leq1$,
\begin{align}\label{ZWMK001}
\|z_{k}\|_{L^{\infty}_{T}L_{x}^{p}\cap L_{T}^{\frac{p}{\alpha}}L_{x}^{\frac{3p}{3-2\alpha}}}+\big\|\nabla|z_{k}|^{\frac{p}{2}}\big\|^{\frac{2}{p}}_{L_{T}^{2}L_{x}^{2}}\leq& C(p,\alpha,c,\gamma)\|w_{1}\otimes w_{2}\|_{L_{T}^{2}L_{x}^{p}}.
\end{align}

{\bf Step 2.} Since $z_{k+1}-z_{k}$ also satisfies equations \eqref{CAU001} with $a_{i}$ replaced by $z_{i}$, $i=k-1,k,k+1$, then by the same way as in Step 2 of the proof in Lemma \ref{Lem01}, we deduce that $\{z_{k}\}$ is a Cauchy sequence in $L^{\infty}_{T}L_{x}^{p}\cap L_{T}^{\frac{p}{\alpha}}L_{x}^{\frac{3p}{3-2\alpha}}$ for any $0<\alpha\leq1$. Moreover, there exists a limit $z\in L^{\infty}_{T}L_{x}^{p}\cap L_{T}^{\frac{p}{\alpha}}L_{x}^{\frac{3p}{3-2\alpha}}$ such that
\begin{align*}
\lim\limits_{k\rightarrow\infty}\|z_{k}-z\|_{L_{T}^{\infty}L_{x}^{p}\cap L_{T}^{\frac{p}{\alpha}}L_{x}^{\frac{3p}{3-2\alpha}}}+\|\nabla|z_{k}-z|^{\frac{p}{2}}\|^{\frac{2}{p}}_{L^{2}_{T}L^{2}_{x}}=0,
\end{align*}
which, together with \eqref{ZWMK001} and H\"{o}lder's inequality, shows that
\begin{align*}
\|z\|_{L^{\infty}_{T}L_{x}^{p}\cap L_{T}^{\frac{p}{\alpha}}L_{x}^{\frac{3p}{3-2\alpha}}}+\big\|\nabla|z|^{\frac{p}{2}}\big\|^{\frac{2}{p}}_{L_{T}^{2}L_{x}^{2}}\leq& C(p,\alpha,c,\gamma)\|w_{1}\otimes w_{2}\|_{L_{T}^{2}L_{x}^{p}}\notag\\
\leq&C(p,\alpha,c,\gamma)
\begin{cases}
\|w_{1}\|_{L_{T}^{\frac{4p}{2p-3}}L_{x}^{2p}}\|w_{2}\|_{L_{T}^{\frac{4p}{3}}L_{x}^{2p}},\\
T^{\frac{p-3}{2p}}\|w_{1}\|_{L_{T}^{\frac{4p}{3}}L_{x}^{2p}}\|w_{2}\|_{L_{T}^{\frac{4p}{3}}L_{x}^{2p}}.
\end{cases}
\end{align*}
Following Step 3 in the proof of Lemma \ref{Lem01} with minor modification, we also have $z\in C([0,T];L^{p}(\mathbb{R}^{3}))$. The proof is complete.

\end{proof}

We are now ready to prove Theorems \ref{THM001} and \ref{THM002} by combining Lemmas \ref{PER001}, \ref{Lem01} and \ref{Lem02}.
\begin{proof}[Proofs of Theorems \ref{THM001} and \ref{THM002}]
We divide the proofs into two parts as follows.

{\bf Part 1.} Let $p\geq3$. For $w_{0}\in L^{p}(\mathbb{R}^{3})$ and $w_{i}\in L^{\frac{4p}{3}}([0,T];L^{2p}(\mathbb{R}^{3}))$, $i=1,2$, it follows from Lemmas \ref{Lem01} and \ref{Lem02} that
\begin{align}\label{HME001}
\begin{cases}
\|a\|_{L_{T}^{\frac{4p}{3}}L_{x}^{2p}}\leq  C_{1}\|w_{0}\|_{L^{p}},\\
\|N(w_{1},w_{2})\|_{L_{T}^{\frac{4p}{3}}L_{x}^{2p}}\leq C_{2} T^{\frac{p-3}{2p}}\|w_{1}\|_{L_{T}^{\frac{4p}{3}}L_{x}^{2p}}\|w_{2}\|_{L_{T}^{\frac{4p}{3}}L_{x}^{2p}}.
\end{cases}
\end{align}
Then applying Lemma \ref{PER001} with $E=L_{T}^{\frac{4p}{3}}L_{x}^{2p}$, there exists a small constant $T>0$ such that $\|a\|_{L_{T}^{\frac{4p}{3}}L_{x}^{2p}}T^{\frac{p-3}{2p}}<\frac{1}{4C_{2}}$, and there thus exists a unique local $L^{p}$ mild solution $w$ to problem \eqref{NS005} on $[0,T]$ with $w$ satisfying \eqref{AQ005}.

{\bf Part 2.} Using \eqref{HME001} with $p=3$, we pick $0<\varepsilon_{0}\leq\frac{1}{4C_{1}C_{2}}$ such that $\|w_{0}\|_{L^{3}}<\varepsilon_{0}$. Then for any $T>0$, by applying Lemma \ref{PER001} with $E=L_{T}^{4}L_{x}^{6}$, we obtain that problem \eqref{NS005} has a unique $L^{p}$ mild solution $w$ on $[0,T]$ satisfying
\begin{align}\label{EQUNA01}
\|w\|_{C_{t}L_{x}^{3}\cap L_{t}^{4}L_{x}^{6}}+\|\nabla|w|^{\frac{3}{2}}\|^{\frac{2}{3}}_{L^{2}_{t}L^{2}_{x}}\leq C(p,c,\gamma)\|w_{0}\|_{L^{3}(\mathbb{R}^{3})}.
\end{align}
Since $T>0$ is arbitrary, $w$ is a global-in-time solution.

Consider the case when $w_{0}\in L_{\sigma}^{p}\cap L_{\sigma}^{3}(\mathbb{R}^{3})$ and $\|w_{0}\|_{L^{3}}<\varepsilon_{0}$, where $3<p\leq\frac{9}{2}$. Utilizing the Gagliardo-Nirenberg interpolation inequality with \eqref{EQUNA01}, we deduce that for $3<p\leq\frac{9}{2}$,
\begin{align*}
\|w\|_{L_{t}^{\frac{4p}{2p-3}}L_{x}^{2p}}\leq C\|w\|_{L_{t}^{\infty}L_{x}^{3}}^{\frac{9-2p}{4p}}\|\nabla|w|^{\frac{3}{2}}\|_{L_{t}^{2}L_{x}^{2}}^{\frac{2p-3}{2p}}\leq C\|w_{0}\|_{L^{3}(\mathbb{R}^{3})}.
\end{align*}
This, in combination with Lemmas \ref{Lem01} and \ref{Lem02}, reads that
\begin{align*}
\|w\|_{C_{t}L_{x}^{p}\cap L_{t}^{\frac{4p}{3}}L_{x}^{2p}}\leq& \|a\|_{C_{t}L_{x}^{p}\cap L_{t}^{\frac{4p}{3}}L_{x}^{2p}}+\|N(w,w)\|_{C_{t}L_{x}^{p}\cap L_{t}^{\frac{4p}{3}}L_{x}^{2p}}\notag\\
\leq&C\|w_{0}\|_{L^{p}}+C\|w\|_{L_{t}^{\frac{4p}{2p-3}}L_{x}^{2p}}\|w\|_{L_{t}^{\frac{4p}{3}}L_{x}^{2p}}\notag\\
\leq&C\|w_{0}\|_{L^{p}}+C\varepsilon_{0}\|w\|_{L_{t}^{\frac{4p}{3}}L_{x}^{2p}}.
\end{align*}
By decreasing $\varepsilon_{0}$ if necessary, we have $C\varepsilon_{0}<1$. Then we obtain
\begin{align*}
\|w\|_{C_{t}L_{x}^{p}\cap L_{t}^{\frac{4p}{3}}L_{x}^{2p}}\leq&C\|w_{0}\|_{L^{p}}.
\end{align*}
Therefore, combining these above results, we see that $w$ is a global $L^{p}$ mild solution for problem \eqref{NS005} under the condition of $w_{0}\in L_{\sigma}^{p}\cap L_{\sigma}^{3}(\mathbb{R}^{3})$ and $\|w_{0}\|_{L^{3}}<\varepsilon_{0}$ with $p\in[3,\frac{9}{2}]$.

\end{proof}

\section{Asymptotic stability for solutions to the perturbed equations}\label{SEC004}

We first introduce the logarithmic Sobolev inequality which will be used in the following. Its proof can be seen in Theorem 8.14 of \cite{LL2001}.
\begin{lemma}\label{LEM889}
For any $f\in \dot{H}^{1}(\mathbb{R}^{3})$ and $a>0$, there holds
\begin{align*}
\int_{\mathbb{R}^{3}}f^{2}\ln\frac{f^{2}}{\|f\|_{2}^{2}}dx+3(1+\ln a)\|u\|_{2}^{2}\leq\frac{a^{2}}{\pi}\int_{\mathbb{R}^{3}}|\nabla f|^{2}.
\end{align*}

\end{lemma}

We are now ready to give the proof of Theorem \ref{THM003}. Remark that the method used in Theorem \ref{THM003} can be trace back to previous work \cite{CL1995} and has been applied to the study for asymptotic stability of solutions to the perturbed problem of Landau solutions to the Navier-Stokes equations in recent work \cite{LZZ2023}.
\begin{proof}[Proof of Theorem \ref{THM003}]
We divide into two substeps to complete the proof.

{\bf Step 1.} Consider the case when $q>3$. Set $T>0$ and denote $r(t)=\frac{1}{\frac{1}{T}(\frac{1}{q}-\frac{1}{3})t+\frac{1}{3}}$, $t\in[0,T]$. By the standard approximation argument (for example, see \cite{KT2020,LR2002,LR2016}), we assume without loss of generality that $w$ is smooth. In light of \eqref{NS005}, it follows from a direct calculation that
\begin{align}\label{MGEQ001}
&r(t)^{2}\|w(\cdot,t)\|_{r(t)}^{r(t)-1}\frac{d}{dt}\|w(\cdot,t)\|_{r(t)}\notag\\
&=\dot{r}(t)\int_{\mathbb{R}^{3}}|w(\cdot,t)|^{r(t)}\ln\big(|w(\cdot,t)|^{r(t)}/\|w(\cdot,t)\|_{r(t)}^{r(t)}\big)dx\notag\\
&\quad+r(t)^{2}\int_{\mathbb{R}^{3}}|w(\cdot,t)|^{r(t)-2}w_{i}\partial_{j}\partial_{j}w_{i}dx\notag\\
&\quad+r(t)^{2}\int_{\mathbb{R}^{3}}|w(\cdot,t)|^{r(t)-2}w_{i}[\partial_{j}(w_{i}w_{j}+w_{i}u^{c,\gamma}_{j}+u^{c,\gamma}_{i}w_{j})-\partial_{i}\pi]dx.
\end{align}
For the second term on the right-hand side of \eqref{MGEQ001}, we have from integration by parts that
\begin{align}\label{MGEQ002}
&r(t)^{2}\int_{\mathbb{R}^{3}}|w(\cdot,t)|^{r(t)-2}w_{i}\partial_{j}\partial_{j}w_{i}dx=-r(t)^{2}\int_{\mathbb{R}^{3}}\partial_{j}(|w(\cdot,t)|^{r(t)-2}w_{i})\partial_{j}w_{i}dx\notag\\
&=-\frac{r(t)^{2}}{2}\int_{\mathbb{R}^{3}}\nabla|w(t)|^{r(t)-2}\nabla|w|^{2}dx-r(t)^{2}\int_{\mathbb{R}^{3}}|w(t)|^{r(t)-2}|\nabla w|^{2}dx\notag\\
&=-4(r(t)-2)\big\|\nabla|w(t)|^{\frac{r(t)}{2}}\big\|_{L^{2}}^{2}-r(t)^{2}\int_{\mathbb{R}^{3}}|w(t)|^{r(t)-2}|\nabla w|^{2}dx.
\end{align}
%We here point out that the constant in the front of $\big\|\nabla|w(t)|^{\frac{r(t)}{2}}\big\|_{L^{2}}^{2}$ is $-4(r(t)-2)$ rather than $-4r(t)(r(t)-2)$ as appeared in line 3 of page 15 in \cite{LZZ2023}.
With regard to the last term on the right-hand side of \eqref{MGEQ001}, we have the following result. Claim that
\begin{align}\label{FN006}
&\left|r(t)^{2}\int_{\mathbb{R}^{3}}|w(\cdot,t)|^{r(t)-2}w_{i}[\partial_{j}(w_{i}w_{j}+w_{i}u^{c,\gamma}_{j}+u^{c,\gamma}_{i}w_{j})-\partial_{i}\pi]dx\right|\notag\\
&\leq C(q,c,\gamma)r^{2}(K(c,\gamma)+\|w_{0}\|_{L^{3}})\|\nabla|w|^{\frac{r}{2}}\|_{L^{2}}^{2}.
\end{align}
For simplicity, denote
\begin{align*}
\mathcal{A}_{1}=&r(t)^{2}\int_{\mathbb{R}^{3}}|w(\cdot,t)|^{r(t)-2}w_{i}\partial_{j}(w_{i}w_{j})dx,\notag\\
\mathcal{A}_{2}=&r(t)^{2}\int_{\mathbb{R}^{3}}|w(\cdot,t)|^{r(t)-2}w_{i}\partial_{j}(w_{i}u^{c,\gamma}_{j})dx,\notag\\
\mathcal{A}_{3}=&r(t)^{2}\int_{\mathbb{R}^{3}}|w(\cdot,t)|^{r(t)-2}w_{i}\partial_{j}(u^{c,\gamma}_{i}w_{j})dx,\notag\\
\mathcal{A}_{4}=&r(t)^{2}\int_{\mathbb{R}^{3}}|w(\cdot,t)|^{r(t)-2}w_{i}\partial_{i}\pi dx.
\end{align*}
As for the first term $\mathcal{A}_{1}$, we have from integration by parts, H\"{o}lder's inequality, Sobolev embedding $\dot{H}^{1}(\mathbb{R}^{3})\hookrightarrow L^{6}(\mathbb{R}^{3})$ (see \cite{S1992} for the best constant) and \eqref{EQUNA01} that
\begin{align*}
\mathcal{A}_{1}=&-r(t)^{2}\int_{\mathbb{R}^{3}}\nabla|w(\cdot,t)|^{r(t)-2}w|w|^{2}dx\notag\\
\leq&r(t)(r(t)-2)\int_{\mathbb{R}^{3}}\big|\nabla|w(\cdot,t)|^{\frac{r(t)}{2}}\big||w|^{\frac{r(t)}{2}}|w|dx\notag\\
\leq&r(t)(r(t)-2)\big\|\nabla|w|^{\frac{r(t)}{2}}\big\|_{L^{2}}\big\||w|^{\frac{r(t)}{2}}\big\|_{L^{6}}\|w\|_{L^{3}}\notag\\
\leq&C_{0}r(t)(r(t)-2)\|w_{0}\|_{L^{3}}\big\|\nabla|w|^{\frac{r(t)}{2}}\big\|^{2}_{L^{2}}.
\end{align*}
For the second term $\mathcal{A}_{2}$, a consequence of integration by parts, H\"{o}lder's inequality, Corollary \ref{MZLEM001} and Lemma \ref{CORO001} gives that
\begin{align*}
\mathcal{A}_{2}=&-\frac{r(t)^{2}}{2}\int_{\mathbb{R}^{3}}u^{c,\gamma}\nabla|w(\cdot,t)|^{r(t)-2}|w|^{2}dx\notag\\
=&-r(t)(r(t)-2)\int_{\Omega_{e^{-1}}}(|x'||x|)^{\frac{1}{2}}u^{c,\gamma}\nabla|w(\cdot,t)|^{\frac{r(t)}{2}}\frac{|w(\cdot,t)|^{\frac{r(t)}{2}}}{(|x'||x|)^{\frac{1}{2}}}dx\notag\\
&-r(t)(r(t)-2)\int_{\mathbb{R}^{3}\setminus\Omega_{e^{-1}}}|x|u^{c,\gamma}\nabla|w(\cdot,t)|^{\frac{r(t)}{2}}\frac{|w(\cdot,t)|^{\frac{r(t)}{2}}}{|x|}dx\notag\\
\leq&r(t)(r(t)-2)K(c,\gamma)\big\|\nabla|w|^{\frac{r(t)}{2}}\big\|_{L^{2}}\left(\bigg\|\frac{|w|^{\frac{r(t)}{2}}}{(|x'||x|)^{\frac{1}{2}}}\bigg\|_{L^{2}}+\bigg\|\frac{|w|^{\frac{r(t)}{2}}}{|x|}\bigg\|_{L^{2}}\right)\notag\\
\leq&C_{0}r(t)(r(t)-2)K(c,\gamma)\big\|\nabla|w|^{\frac{r(t)}{2}}\big\|^{2}_{L^{2}}.
\end{align*}
Recall the following elementary inequalities that

$(1)$ for real numbers $a_{i}$, $i=1,...,n$, $n\geq1$,
\begin{align}\label{NUM001}
\left(\sum^{n}_{i=1}a_{i}\right)^{2}\leq n\sum^{n}_{i=1}a_{i}^{2};
\end{align}

$(2)$ for $c,d\geq0$, $c+d\geq1$, $1\leq m\leq\infty$, $x_{i},y_{i}\geq0$, $1\leq i\leq m,$
\begin{align}\label{EQU02}
\sum^{m}_{i=1}x_{i}^{c}y_{i}^{d}\leq\left(\sum^{m}_{i=1}x_{i}\right)^{c}\left(\sum^{m}_{i=1}y_{i}\right)^{d}.
\end{align}
With regard to the third term $\mathcal{A}_{3}$, it follows from Corollary \ref{MZLEM001}, \eqref{Z02}, \eqref{NUM001} and \eqref{EQU02} that
\begin{align*}
\mathcal{A}_{3}\leq&9r(t)^{2}\int_{\mathbb{R}^{3}}\frac{|w(\cdot,t)|^{r(t)}}{|x'||x|}(|x'||x||\nabla u^{c,\gamma}|)dx\notag\\
\leq&9r(t)^{2}K(c,\gamma)\big\|\nabla|w|^{\frac{r(t)}{2}}\big\|^{2}_{L^{2}}.
\end{align*}
Remark that we use \eqref{Z02} rather than \eqref{Z01} to achieve a direct and simple computation for the term $\mathcal{A}_{3}$.

We now proceed to estimate the last term $\mathcal{A}_{4}$. Since $\pi=(-\Delta)^{-1}\partial_{i}\partial_{j}(u^{c,\gamma}_{i}w_{j}+w_{i}u^{c,\gamma}_{j}+w_{i}w_{j})$, we obtain from integration by parts that
\begin{align*}
\mathcal{A}_{4}=&r(t)^{2}\int_{\mathbb{R}^{3}}\partial_{i}|w(\cdot,t)|^{r(t)-2}w_{i}\pi dx\notag\\
=&r(t)^{2}\int_{\mathbb{R}^{3}}\partial_{i}|w(\cdot,t)|^{r(t)-2}w_{i}(-\Delta)^{-1}\partial_{i}\partial_{j}(u^{c,\gamma}_{i}w_{j}+w_{i}u^{c,\gamma}_{j})dx\notag\\
&+r(t)^{2}\int_{\mathbb{R}^{3}}\partial_{i}|w(\cdot,t)|^{r(t)-2}w_{i}(-\Delta)^{-1}\partial_{i}\partial_{j}(w_{i}w_{j})dx\notag\\
=&:\mathcal{A}_{4}^{1}+\mathcal{A}_{4}^{2}.
\end{align*}
Similar to \eqref{ZZW006}, since $(|x'||x|)^{\frac{r-2}{2}}$ and $|x|^{r-2}$ are all $A_{r}$-weights, it follows from Theorem 9.4.6 in \cite{G2009}, H\"{o}lder's inequality, Corollary \ref{MZLEM001} and Lemma \ref{CORO001} that
\begin{align*}
\mathcal{A}_{4}^{1}\leq&2r(t)(r(t)-2)\int_{\mathbb{R}^{3}}|\nabla|w(\cdot,t)|^{\frac{r(t)}{2}}||w(\cdot,t)|^{\frac{r(t)}{2}-1}|(-\Delta)^{-1}\partial_{i}\partial_{j}(u^{c,\gamma}_{i}w_{j}+w_{i}u^{c,\gamma}_{j})|dx\notag\\
\leq&4r(t)(r(t)-2)C_{r}\|(|x'||x|)^{\frac{r-2}{2r}}(u^{c,\gamma}\otimes w)\|_{L^{r}(\Omega_{e^{-1}})}\bigg\|\frac{|w|^{\frac{r}{2}-1}}{(|x'||x|)^{\frac{r-2}{2r}}}\bigg\|_{L^{\frac{2r}{r-2}}}\|\nabla |w|^{\frac{r}{2}}\|_{L^{2}}\notag\\
&+4r(t)(r(t)-2)C_{r}\||x|^{\frac{r-2}{r}}(u^{c,\gamma}\otimes w)\|_{L^{r}(\mathbb{R}^{3}\setminus\Omega_{e^{-1}})}\bigg\|\frac{|w|^{\frac{r}{2}-1}}{|x|^{\frac{r-2}{r}}}\bigg\|_{L^{\frac{2r}{r-2}}}\|\nabla |w|^{\frac{r}{2}}\|_{L^{2}}\notag\\
\leq&4r(t)(r(t)-2)C_{r}K(c,\gamma)\left(\bigg\|\frac{|w|^{\frac{r}{2}}}{(|x'||x|)^{\frac{1}{2}}}\bigg\|_{L^{2}}+\bigg\|\frac{|w|^{\frac{r}{2}}}{|x|}\bigg\|_{L^{2}}\right)\|\nabla |w|^{\frac{r}{2}}\|_{L^{2}}\notag\\
\leq&4\overline{C}r(t)(r(t)-2)C_{r}K(c,\gamma)\|\nabla |w|^{\frac{r}{2}}\|^{2}_{L^{2}},
\end{align*}
where $\overline{C}$ is the embedding constant and $C_{r}$ is given in Theorem 9.4.6 of \cite{G2009}. As for $\mathcal{A}_{4}^{2}$, we deduce from \eqref{QMYU001}, H\"{o}lder's inequality and Sobolev embedding $\dot{H}_{1}(\mathbb{R}^{3})\hookrightarrow L^{6}(\mathbb{R}^{3})$ that
\begin{align*}
\mathcal{A}_{4}^{2}\leq&2r(t)(r(t)-2)\big\|\nabla|w(\cdot,t)|^{\frac{r(t)}{2}}\big\|_{L^{2}}\big\||w(\cdot,t)|^{\frac{r(t)}{2}-1}\big\|_{L^{\frac{6r}{r-2}}}\|(-\Delta)^{-1}\partial_{i}\partial_{j}w_{i}w_{j}\|_{L^{\frac{3r}{r+1}}}\notag\\
\leq&2r(t)(r(t)-2)\overline{C}_{0}H_{\frac{3r}{r+1}}\big\|\nabla|w(\cdot,t)|^{\frac{r(t)}{2}}\big\|_{L^{2}}\big\||w(\cdot,t)|^{\frac{r(t)}{2}-1}\big\|_{L^{\frac{6r}{r-2}}}\|w\otimes w\|_{L^{\frac{3r}{r+1}}}\notag\\
\leq&2r(t)(r(t)-2)\overline{C}_{0}H_{\frac{3r}{r+1}}\big\|\nabla|w(\cdot,t)|^{\frac{r(t)}{2}}\big\|_{L^{2}}\big\|w(\cdot,t)\big\|_{L^{3r}}^{\frac{r(t)}{2}-1}\|w(\cdot,t)\|_{L^{3r}}\|w(\cdot,t)\|_{L^{3}}\notag\\
\leq&2r(t)(r(t)-2)\overline{C}_{0}H_{\frac{3r}{r+1}}\big\|\nabla|w(\cdot,t)|^{\frac{r(t)}{2}}\big\|_{L^{2}}\big\||w(\cdot,t)|^{\frac{r(t)}{2}}\big\|_{L^{6}}\|w(\cdot,t)\|_{L^{3}}\notag\\
\leq&Cr(t)(r(t)-2)\|w_{0}\|_{L^{3}}\big\|\nabla|w(\cdot,t)|^{\frac{r(t)}{2}}\big\|_{L^{2}}^{2}.
\end{align*}
Combining these above results, we obtain that \eqref{FN006} holds.

Once \eqref{FN006} is proved, for any $q>3$ and $1<\tau<1$, we can choose two sufficiently small positive constants $\delta=\delta(q,\tau)\leq\delta_{0}$ and $\varepsilon=\varepsilon(q,\tau)\leq\varepsilon_{0}$, where $\delta_{0}$ and $\varepsilon_{0}$ are given in Theorem \ref{THM002}, such that if $(c,\gamma)\in M\cap\{|(c,\gamma)|\leq\delta\}$ and $\|w_{0}\|_{L^{3}}<\varepsilon$, then
\begin{align}\label{ARQK002}
Cr^{2}(K(c,\gamma)+\varepsilon)\leq 4\tau(r-2).
\end{align}
Therefore, substituting \eqref{MGEQ002}, \eqref{FN006} and \eqref{ARQK002} into \eqref{MGEQ001}, we have
\begin{align*}
&r(t)^{2}\|w(\cdot,t)\|_{r(t)}^{r(t)-1}\frac{d}{dt}\|w(\cdot,t)\|_{r(t)}\notag\\
&\leq\dot{r}(t)\int_{\mathbb{R}^{3}}|w(\cdot,t)|^{r(t)}\ln\big(|w(\cdot,t)|^{r(t)}/\|w(\cdot,t)\|_{r(t)}^{r(t)}\big)dx\notag\\
&\quad-4(r(t)-2)(1-\tau)\|\nabla|w|^{\frac{r}{2}}\|_{L^{2}}^{2},
\end{align*}
which, together with Lemma \ref{LEM889}, shows that for $a>0,$
\begin{align*}
&r(t)^{2}\|w(\cdot,t)\|_{r(t)}^{r(t)-1}\frac{d}{dt}\|w(\cdot,t)\|_{r(t)}\notag\\
&\leq\left(\frac{a^{2}}{\pi}\dot{r}-4(r(t)-2)(1-\tau)\right)\|\nabla|w|^{\frac{r}{2}}\|_{L^{2}}^{2}-3\dot{r}(1+\ln a)\big\||w|^{\frac{r}{2}}\big\|_{L^{2}}^{2}.
\end{align*}
By choosing $a=\big(\frac{4\pi(r(t)-2)(1-\tau)}{\dot{r}}\big)^{\frac{1}{2}}$, we have
\begin{align*}
&r(t)^{2}\|w(\cdot,t)\|_{r(t)}^{r(t)-1}\frac{d}{dt}\|w(\cdot,t)\|_{r(t)}\leq-3\dot{r}(1+\ln a)\big\||w|^{\frac{r}{2}}\big\|_{L^{2}}^{2},
\end{align*}
which reads that
\begin{align*}
\frac{d}{dt}\|w(t)\|_{r(t)}\leq&\bigg[-\frac{1}{T}\left(\frac{1}{3}-\frac{1}{q}\right)\left(3+\frac{3}{2}\ln\frac{4\pi(r(t)-2)(1-\tau)}{r^{2}}\right)\notag\\
&\;+\frac{3}{2T}\left(\frac{1}{3}-\frac{1}{q}\right)\ln\frac{1}{T}\left(\frac{1}{3}-\frac{1}{q}\right)\bigg]\|w(t)\|_{r(t)}.
\end{align*}
By Gronwall's inequality, we deduce
\begin{align}\label{DECON06}
\|w(T)\|_{L^{q}}\leq\mathcal{C}_{q}\left(\frac{1}{3}-\frac{1}{q}\right)^{\frac{3}{2}(\frac{1}{3}-\frac{1}{q})}T^{-\frac{3}{2}(\frac{1}{3}-\frac{1}{q})}\|w_{0}\|_{L^{3}},
\end{align}
where
\begin{align*}
\mathcal{C}_{q}=e^{-\frac{1}{T}(\frac{1}{3}-\frac{1}{q})\int_{0}^{T}(3+\frac{3}{2}\ln\frac{4\pi(r(t)-2)(1-\tau)}{r(t)^{2}})dt}.
\end{align*}
We proceed to give a precise calculation in terms of the value of $\mathcal{C}_{q}$. Observe that
\begin{align*}
\int_{0}^{T}\ln\frac{4\pi(r(t)-2)(1-\tau)}{r(t)^{2}}dt=\frac{T}{\frac{1}{3}-\frac{1}{q}}\int_{3}^{q}\frac{1}{r^{2}}\ln\frac{4\pi(r-2)(1-\tau)}{r^{2}}dr,
\end{align*}
and
\begin{align*}
&\int_{3}^{q}\frac{1}{r^{2}}\ln\frac{4\pi(r-2)(1-\tau)}{r^{2}}dr\notag\\
&=\int_{3}^{q}\frac{\ln4\pi(1-\tau)}{r^{2}}dr+\int_{3}^{q}\frac{\ln(r-2)}{r^{2}}dr-\int_{3}^{q}\frac{\ln r^{2}}{r^{2}}dr.
\end{align*}
By a straight-forward computation, we have
\begin{align*}
&\int_{3}^{q}\frac{\ln4\pi(1-\tau)}{r^{2}}dr=\left(\frac{1}{3}-\frac{1}{q}\right)\ln4\pi(1-\tau),\\
&\int_{3}^{q}\frac{\ln(r-2)}{r^{2}}dr=\left(\frac{1}{2}-\frac{1}{q}\right)\ln(q-2)+\frac{1}{2}\ln\frac{3}{q},\\
&\int_{3}^{q}\frac{\ln r^{2}}{r^{2}}dr=2\left(\frac{1+\ln3}{3}-\frac{1+\ln q}{q}\right).
\end{align*}
Combining these above computational results, we obtain
\begin{align*}
\mathcal{C}_{q}=3^{-\frac{7}{4}}q^{\frac{3}{q}}\left(q-2\right)^{\frac{3}{2q}}\left(\frac{q}{q-2}\right)^{\frac{3}{4}}(4\pi(1-\tau))^{-\frac{3}{2}(\frac{1}{3}-\frac{1}{q})}e^{-6(\frac{1}{3}-\frac{1}{q})}.
\end{align*}

{\bf Step 2.} Consider the case of $q=3$. Since $w_{0}\in L_{\sigma}^{3}$ and $L_{\sigma}^{\frac{5}{2}}\cap L_{\sigma}^{3}$ is dense in $L_{\sigma}^{3}$, then there exists a sequence $\{w_{0,k}\}$ in $L_{\sigma}^{\frac{5}{2}}\cap L_{\sigma}^{3}$ such that
\begin{align}\label{CONVERG}
\lim_{k\rightarrow\infty}\|w_{0,k}-w_{0}\|_{L^{3}}=0.
\end{align}
This, together with the assumed condition that $\|w_{0}\|_{L^{3}}<\frac{\varepsilon_{0}}{2}$, leads to that there exists some $N_{0}=N(\varepsilon_{0})$ such that if $k\geq N_{0}$, we have $\|w_{0,k}-w_{0}\|_{L^{3}}<\frac{\varepsilon_{0}}{2}$. Then we obtain that for $k\geq N_{0}$,
\begin{align*}
\|w_{0,k}\|_{L^{3}}\leq\|w_{0,k}-w_{0}\|_{L^{3}}+\|w_{0}\|_{L^{3}}<\varepsilon_{0}.
\end{align*}
Then using Theorem \ref{THM002}, for every $k\geq N_{0}$, there exists a unique global $L^{3}$ mild solution $w_{k}$ to equations \eqref{NS005} with the initial data $w_{0,k}\in L_{\sigma}^{3}\cap L_{\sigma}^{\frac{5}{2}}(\mathbb{R}^{3})$ and $\|w_{0,k}\|_{L^{3}}<\varepsilon_{0}$. For any $T>0,$ by picking $r_{k}(t)=\frac{1}{-\frac{t}{15T}+\frac{2}{5}}$ on $[0,T]$, it then follows from the proof of \eqref{DECON06} with a slight modification that for $T>0$,
\begin{align}\label{ZWK001}
\|w_{k}(T)\|_{L^{3}}\leq C_{0}T^{-\frac{1}{10}}\|w_{0,k}\|_{L^{\frac{5}{2}}}\rightarrow0,\quad\text{as $T\rightarrow\infty$},
\end{align}
where $C_{0}$ is a universal positive constant. It is worth emphasizing that we adopt the norm $\|w_{0,k}\|_{L^{\frac{5}{2}}}$ rather than $\|w_{0,k}\|_{L^{2}}$ on the right-hand side of \eqref{ZWK001} for the purpose of ensuring that condition \eqref{ARQK002} holds.

For $k\geq N_{0}$, let $v_{k}=w-w_{k}$ and $\pi_{k}=p-p_{k}$. For any $T>0$, $v_{k}$ is a $L^{3}$ mild solution in $\mathbb{R}^{3}\times(0,T)$ of the following equations
\begin{align}\label{EQUAE001}
\begin{cases}
\partial_{t}v_{k}-\Delta v_{k}=\mathrm{div}(-v_{k}\otimes v_{k}+v_{k}\otimes w+w\otimes v_{k})\\
+\mathrm{div}(u^{c,\gamma}\otimes v_{k}+v_{k}\otimes u^{c,\gamma})+\nabla\pi_{k},\\
\mathrm{div}v_{k}=0,\\
v_{k}(x,0)=w_{0}(x)-w_{0,k}(x),
\end{cases}
\end{align}
where
\begin{align}\label{WADEQ001}
\|v_{k}(\cdot,0)\|_{L^{3}}<\frac{\varepsilon_{0}}{2},\quad\text{for $k\geq N_{0}$}.
\end{align}
Applying the proofs of Lemmas \ref{Lem01} and \ref{Lem02} to equations \eqref{EQUAE001} with minor modification, we obtain
\begin{align}\label{DEAQ001}
&\|v_{k}\|_{C_{T}L_{x}^{3}\cap L_{T}^{4}L_{x}^{6}}+\big\|\nabla|v_{k}|^{\frac{3}{2}}\big\|^{\frac{2}{3}}_{L_{T}^{2}L_{x}^{2}}\notag\\
&\leq C_{1}(\|v_{k}(\cdot,0)\|_{L^{3}}+\|v_{k}\|_{L_{T}^{4}L_{x}^{6}}^{2})+C_{2}\left(\int_{0}^{T}\|v_{k}\|_{L^{6}}^{2}\|w\|_{L^{6}}^{2}dt\right)^{\frac{1}{2}}.
\end{align}
Observe from the interpolation inequality that $\|v_{k}\|_{L^{6}}\leq\|v_{k}\|^{\frac{1}{4}}_{L^{3}}\|v_{k}\|^{\frac{3}{4}}_{L^{9}}$. This, together with H\"{o}lder's inequality and Young's inequality, reads that for $\mu>0$,
\begin{align*}
\left(\int_{0}^{T}\|v_{k}\|_{L^{6}}^{2}\|w\|_{L^{6}}^{2}dt\right)^{\frac{1}{2}}\leq&\|v_{k}\|_{L_{T}^{3}L_{x}^{9}}^{\frac{3}{4}}\left(\int_{0}^{T}\|v_{k}\|_{L^{3}}\|w\|_{L^{6}}^{4}dt\right)^{\frac{1}{4}}\notag\\
\leq&\mu\|v_{k}\|_{L_{T}^{3}L_{x}^{9}}+\frac{27}{256\mu^{3}}\int_{0}^{T}\|v_{k}\|_{L^{3}}\|w\|_{L^{6}}^{4}dt.
\end{align*}
Substituting this into \eqref{DEAQ001}, it follows from Sobolev embedding $\dot{H}^{1}(\mathbb{R}^{3})\hookrightarrow L^{6}(\mathbb{R}^{3})$ that
\begin{align*}
&\|v_{k}\|_{C_{T}L_{x}^{3}\cap L_{t}^{4}L_{x}^{6}}+\big\|\nabla|v_{k}|^{\frac{3}{2}}\big\|^{\frac{2}{3}}_{L_{T}^{2}L_{x}^{2}}\notag\\
&\leq C_{1}\big(\|v_{k}(\cdot,0)\|_{L^{3}}+\|v_{k}\|_{L_{T}^{4}L_{x}^{6}}^{2}\big)+\widetilde{C}_{2}\mu\|\nabla|v_{k}|^{\frac{3}{2}}\|_{L_{T}^{2}L_{x}^{2}}^{\frac{2}{3}}+\frac{C_{2}}{\mu^{3}}\int_{0}^{T}\|v_{k}\|_{L^{3}}\|w\|_{L^{6}}^{4}dt.
\end{align*}
Picking $\widetilde{C}_{2}\mu=\frac{1}{2}$, it follows from Gronwall's inequality that
\begin{align}\label{FEQ001}
&\|v_{k}\|_{C_{T}L_{x}^{3}\cap L_{T}^{4}L_{x}^{6}}+\big\|\nabla|v_{k}|^{\frac{3}{2}}\big\|^{\frac{2}{3}}_{L_{T}^{2}L_{x}^{2}}\notag\\
&\leq C_{1}\big(\|v_{k}(\cdot,0)\|_{L^{3}}+\|v_{k}\|_{L_{T}^{4}L_{x}^{6}}^{2}\big)e^{C_{2}\int_{0}^{T}\|w\|_{L^{6}}^{4}dt}.
\end{align}
In light of \eqref{EQUNA01} and \eqref{WADEQ001}, decreasing $\varepsilon_{0}$ if necessary, we have
\begin{align}\label{COD01}
\|v_{k}(\cdot,0)\|_{L^{3}}\leq& \frac{\varepsilon_{0}}{2}\leq \left(4C_{1}^{2}e^{2C_{2}C(c,\gamma)\varepsilon_{0}^{4}}\right)^{-1}\leq \left(4C_{1}^{2}e^{2C_{2}C(c,\gamma)\|w_{0}\|^{4}_{L^{3}}}\right)^{-1}\notag\\
\leq&\left(4C_{1}^{2}e^{2C_{2}\int_{0}^{T}\|w\|_{L^{6}}^{4}dt}\right)^{-1}.
\end{align}
Claim that for any $T>0$,
\begin{align}\label{COND02}
&\|v_{k}\|_{C_{T}L_{x}^{3}\cap L_{t}^{4}L_{x}^{6}}+\big\|\nabla|v_{k}|^{\frac{3}{2}}\big\|^{\frac{2}{3}}_{L_{T}^{2}L_{x}^{2}}\leq2C_{1}\|v_{k}(\cdot,0)\|_{L^{3}}e^{C_{2}\int_{0}^{T}\|w\|_{L^{6}}^{4}dt}.
\end{align}
Instead of using the continuity method (as pointed out in page 25 of \cite{LZZ2023}) to prove that \eqref{COND02} holds, we here provide a new and simple proof. For simplicity, denote
\begin{align*}
g(T):=&\|v_{k}\|_{C_{T}L_{x}^{3}\cap L_{T}^{4}L_{x}^{6}}+\big\|\nabla|v_{k}|^{\frac{3}{2}}\big\|^{\frac{2}{3}}_{L_{T}^{2}L_{x}^{2}},\notag\\
a:=&\|v_{k}(\cdot,0)\|_{L^{3}},\;\, b(T):=e^{C_{2}\int_{0}^{T}\|w\|_{L^{6}}^{4}dt}.
\end{align*}
Then from \eqref{FEQ001}, we have
\begin{align}\label{ZDAE06}
g(T)^{2}-\frac{1}{C_{1}b(T)}g(T)+a\geq0.
\end{align}
Hence by using \eqref{COD01}, we obtain two roots as follows:
\begin{align*}
g_{\pm}(T)=\frac{1\pm\sqrt{1-4ab(T)^{2}C_{1}^{2}}}{2C_{1}b(T)}.
\end{align*}
It suffices to require that $g(T)\leq g_{-}(T)$ or $g(T)\geq g_{+}(T)$ for the purpose of letting \eqref{ZDAE06} hold. We now show that the case of $g(T)\geq g_{+}(T)$ is invalid and must be excluded. In fact, if $g(T)\geq g_{+}(T)$ holds and then let $T\rightarrow 0^{+}$, we have
\begin{align*}
g(0^{+})=&\|v_{k}(\cdot,0)\|_{L^{3}}\geq g_{+}(0^{+})=\frac{1+\sqrt{1-4ab(0)^{2}C_{1}^{2}}}{2C_{1}b(0)}\notag\\
\geq&\frac{2-4 C_{1}^{2}\|v_{k}(\cdot,0)\|_{L^{3}}}{2C_{1}}=\frac{1}{C_{1}}-2C_{1}\|v_{k}(\cdot,0)\|_{L^{3}},
\end{align*}
which, in combination with \eqref{WADEQ001}, reads that for $k\geq N_{0}$,
\begin{align*}
\frac{1}{C_{1}}\leq(2C_{1}+1)\|v_{k}(\cdot,0)\|_{L^{3}}\leq \frac{(2C_{1}+1)\varepsilon_{0}}{2}.
\end{align*}
This leads to a contradiction, since $\varepsilon_{0}$ can be chosen to be small enough such that $\varepsilon_{0}<\frac{2}{C_{1}(2C_{1}+1)}.$ Therefore, we have
\begin{align*}
&\|v_{k}\|_{C_{T}L_{x}^{3}\cap L_{T}^{4}L_{x}^{6}}+\big\|\nabla|v_{k}|^{\frac{3}{2}}\big\|^{\frac{2}{3}}_{L_{T}^{2}L_{x}^{2}}\notag\\
&=g(T)\leq g_{-}(T)=\frac{1-\sqrt{1-4ab(T)^{2}C_{1}^{2}}}{2C_{1}b(T)}\leq\frac{4ab(T)^{2}C_{1}^{2}}{2C_{1}b(T)}\notag\\
&=2C_{1}\|v_{k}(\cdot,0)\|_{L^{3}}e^{C_{2}\int_{0}^{T}\|w\|_{L^{6}}^{4}dt}.
\end{align*}
That is, \eqref{COND02} holds. Sending $T\rightarrow\infty$ and using \eqref{EQUNA01}, we obtain from \eqref{CONVERG} that
\begin{align}\label{AMI008}
\|w-w_{k}\|_{L_{t}^{\infty}([0,\infty);L_{x}^{3})}\leq& 2C_{1}\|v_{k}(\cdot,0)\|_{L^{3}}e^{C_{2}\int_{0}^{\infty}\|w\|_{L^{6}}^{4}dt}\notag\\
\leq& 2C_{1}\|w_{0}-w_{0,k}\|_{L^{3}}e^{C_{0}\|w_{0}\|_{L^{3}}^{4}}\rightarrow0,\quad\text{as $k\rightarrow\infty$}.
\end{align}
Observe that for $k\geq N_{0}$ and $t>0$,
\begin{align*}
\|w(t)\|_{L^{3}}\leq\|w-w_{k}\|_{L_{t}^{\infty}([0,\infty);L_{x}^{3})}+\|w_{k}(t)\|_{L^{3}}.
\end{align*}
Sending $k\rightarrow\infty$ and then $t\rightarrow\infty$, it follows from \eqref{ZWK001} and \eqref{AMI008} that $$\lim\limits_{t\rightarrow\infty}\|w(t)\|_{L^{3}}=0.$$

\end{proof}

\section{The relations between $L^{p}$ mild solutions and $L^{2}$ weak solutions}\label{SEC005}

This section is devoted to making clear the relations between $L^{p}$ mild solutions and $L^{2}$ weak solutions for the perturbed problem \eqref{NS005}. To begin with, we introduce the definition of $L^{2}$ weak solution as follows.
\begin{definition}[$L^{2}$ weak solution]\label{DEFINI06}
For $T>0,$ a vector-valued function $w$ is called a $L^{2}$ weak solution of problem \eqref{NS005}, if
\begin{align}\label{WEA01}
w\in C_{w}([0,T];L^{2}_{\sigma}(\mathbb{R}^{3}))\cap L^{2}([0,T];\dot{H}_{\sigma}^{1}(\mathbb{R}^{3})),
\end{align}
and
\begin{align}\label{WEA02}
&(w(s_{2}),\varphi(s_{2}))+\int_{s_{1}}^{s_{2}}[(\nabla w,\nabla\varphi)+(w\cdot\nabla w+w\cdot\nabla u^{c,\gamma}+u^{c,\gamma}\cdot\nabla w,\varphi)]\notag\\
&=(w(s_{1}),\varphi(s_{1}))+\int^{s_{2}}_{s_{1}}(w,\partial_{t}\varphi)dt
\end{align}
for any $0\leq s_{1}\leq s_{2}\leq T$ and $\varphi\in C([0,\infty);H_{\sigma}^{1}(\mathbb{R}^{3}))\cap C^{1}([0,\infty);L_{\sigma}^{2}(\mathbb{R}^{3}))$, where $(\cdot,\cdot)$ denotes the standard $L^{2}$-inner product. This solution is global if \eqref{WEA01}--\eqref{WEA02} holds for any $0<T<\infty.$
\end{definition}
\begin{remark}
Here we remark that $C_{w}([0,T];L^{2}_{\sigma}(\mathbb{R}^{3}))$ represents the space whose elements consist of weakly continuous $L^{2}(\mathbb{R}^{3})$-valued functions in $t$, that is, for any $t_{0}\in[0,T]$ and $v\in L^{2}(\mathbb{R}^{3})$,
\begin{align*}
\int_{\mathbb{R}^{3}}w(x,t)v(x)dx\rightarrow\int_{\mathbb{R}^{3}}w(x,t_{0})v(x)dx,\quad\mathrm{as}\;t\rightarrow t_{0}.
\end{align*}
\end{remark}
%\begin{remark}
%When $T=\infty$, we can define the corresponding $L^{2}$ weak solution by replacing $[0,T]$ and $0\leq s_{1}\leq s_{2}\leq T$ with $[0,\infty)$ and $0\leq s_{1}\leq s_{2}<\infty$ in Definition \ref{DEFINI06}, respectively.
%\end{remark}

Observe that for $0\leq t_{1}<t_{2}\leq\infty$ and $g\in L_{t}^{2}H_{x}^{1}$, it follows from H\"{o}lder's inequality, Corollary \ref{MZLEM001} and Lemma \ref{CORO001} that
\begin{align}\label{AYU090}
&\left|\int_{t_{1}}^{t_{2}}\int_{\mathbb{R}^{3}}u^{c,\gamma}\cdot(g\cdot\nabla)gdxdt\right|\notag\\
&\leq\int_{t_{1}}^{t_{2}}\big\||x'|^{\frac{1}{2}}|x|^{\frac{1}{2}}u^{c,\gamma}\big\|_{L^{\infty}(\Omega_{e^{-1}})}\big\||x'|^{-\frac{1}{2}}|x|^{-\frac{1}{2}}g\big\|_{L^{2}}\|\nabla g\|_{L^{2}}dt\notag\\
&\quad+\int_{t_{1}}^{t_{2}}\big\||x|u^{c,\gamma}\big\|_{L^{\infty}(\mathbb{R}^{3}\setminus\Omega_{e^{-1}})}\big\||x|^{-1}g\big\|_{L^{2}}\|\nabla g\|_{L^{2}}dt\notag\\
&\leq CK(c,\gamma)\|\nabla g\|_{L_{t}^{2}L_{x}^{2}}^{2},
\end{align}
where $\Omega_{e^{-1}}$ is a cone given by \eqref{CONE999}. Making use of \eqref{AYU090} and following the proofs of Theorem 1.4 in \cite{LZZ2023} and Theorem 9.1 in \cite{T2018} with a slight modification, we obtain a weak-strong uniqueness theorem for problem \eqref{NS005} as follows.
\begin{theorem}\label{THM09}
Choose a small constant $\delta_{0}>0$ such that if $(c,\gamma)\in M\cap\{|(c,\gamma)|\leq\delta_{0}\}$, there holds $K(c,\gamma)<\frac{1}{4}$. For $w_{0}\in L_{\sigma}^{2}(\mathbb{R}^{3})$ and $0<T\leq\infty$, let $u,v$ be $L^{2}$ weak solutions of problem \eqref{NS005} in $\mathbb{R}^{3}\times(0,T)$ satisfying $u(x,0)=v(x,0)=w_{0}$. Assume that $u\in L^{s}([0,T);L^{q}(\mathbb{R}^{3}))$, $\frac{3}{q}+\frac{2}{s}=1$, $q,s\in[2,\infty]$. Then $u\equiv v.$

\end{theorem}
\begin{remark}
The weak-strong uniqueness theorem shows that only if we find a strong solution, each of other weak solutions satisfying the same initial data must be the same to it in the whole spatiotemporal domain.
\end{remark}

The first main result of this section is that a $L^{2}$ weak solution of problem \eqref{NS005} with the initial data belonging to $L_{\sigma}^{2}(\mathbb{R}^{3})$ will become a $L^{p}$ mild solution after some time.
\begin{theorem}\label{LEM05}
Let $(c,\gamma)\in M\cap\{|(c,\gamma)|\leq\delta_{0}\}$ with $\delta_{0}$ given in Theorem \ref{THM09}. Assume that $w$ is a $L^{2}$ weak solution for problem \eqref{NS005} in $\mathbb{R}^{3}\times(0,\infty)$ with the initial data $w_{0}\in L_{\sigma}^{2}(\mathbb{R}^{3}).$ Then for any $3\leq p\leq\frac{9}{2}$, there exists some $t_{0}>0$ such that $w(\cdot+t_{0})$ is a $L^{p}$ mild solution to equations \eqref{NS005} with the initial data $w(t_{0})\in L_{\sigma}^{p}\cap L_{\sigma}^{2}(\mathbb{R}^{3}).$
\end{theorem}
\begin{proof}
Making use of Theorems 1.1 and 1.2 in \cite{LY2021}, we know that for any $0\leq s\leq t$,
\begin{align*}
\|w(t)\|_{L^{2}}^{2}+\int^{t}_{s}\|\nabla\otimes w(\tau)\|_{L^{2}}^{2}d\tau\leq\|w(s)\|_{L^{2}}^{2},\quad\lim\limits_{t\rightarrow\infty}\|w(t)\|_{L^{2}}=0,
\end{align*}
which implies that
\begin{align}\label{QKU001}
\|w(t)\|_{L^{2}}\leq\|w(s)\|_{L^{2}},\quad\|\nabla w\|_{L_{t}^{2}L_{x}^{2}}\leq\|w_{0}\|_{L^{2}}.
\end{align}
By the interpolation inequality, we obtain that for $2<p\leq6,$ $0<\alpha\leq1,$
\begin{align}\label{AZTU001}
\|w\|_{L^{p}}\leq C\|\nabla w\|_{L^{2}}^{\alpha}\|w\|_{L^{2}}^{1-\alpha},\quad\mathrm{with}\;\frac{1}{p}=\frac{1}{2}-\frac{\alpha}{3}.
\end{align}
Integrating \eqref{AZTU001} from $0$ to $\infty$, we have from \eqref{QKU001} that
\begin{align}\label{QKU002}
\|w\|_{L_{t}^{\frac{2}{\alpha}}L_{x}^{p}}\leq &C\|\nabla w\|_{L_{t}^{2}L_{x}^{2}}^{\alpha}\|w_{0}\|_{L^{2}}^{1-\alpha}\leq C\|w_{0}\|_{L^{2}}.
\end{align}
Claim that for any given $\varepsilon>0$, there exists some $t^{\ast}\geq0$ such that for any $2< p\leq6$, there holds $\|w(t^{\ast})\|_{L^{p}}\leq\varepsilon$. If not, there will exist some $\bar{\varepsilon}>0$ and $2< p_{0}\leq6$ such that for all $t\geq0$, $\|w(t)\|_{L^{p_{0}}}\geq\bar{\varepsilon}$. This is a contradiction to \eqref{QKU002}. Therefore, by using \eqref{QKU001} and picking $\varepsilon=\varepsilon_{0}$ with $\varepsilon_{0}$ given by Theorem \ref{THM002}, we find a point $t_{0}\geq0$ such that for $3\leq p\leq\frac{9}{2}$, $w(t_{0})\in L_{\sigma}^{p}\cap L_{\sigma}^{2}\cap L_{\sigma}^{3}(\mathbb{R}^{3})$ and $\|w(t_{0})\|_{L^{3}(\mathbb{R}^{3})}\leq \varepsilon_{0}$. Then Theorem \ref{LEM05} is proved by combining Theorems \ref{THM002} and \ref{THM09}.

\end{proof}

The second main result is that a $L^{p}$ mild solution for problem \eqref{NS005} with the initial data in the class of $L_{\sigma}^{p}\cap L_{\sigma}^{2}(\mathbb{R}^{3})$ also belongs to $L^{2}$ weak solution.
\begin{theorem}\label{DWU08}
For $p\geq3$, let $w$ the unique $L^{p}$ mild solution of equations \eqref{NS005} obtained in Theorem \ref{THM001} with the initial data $w_{0}\in L_{\sigma}^{p}\cap L_{\sigma}^{2}(\mathbb{R}^{3})$ on $[0,T]$ for some $T>0$. Then $w$ is a $L^{2}$ weak solution for problem \eqref{NS005} on $[0,T]$.
\end{theorem}
\begin{remark}
For $p\in[3,\frac{9}{2}]$, let $w$ be the global $L^{p}$ mild solution of equations \eqref{NS005} obtained in Theorem \ref{THM002} under the initial data $w_{0}\in L^{p}_{\sigma}(\mathbb{R}^{3})\cap L^{3}_{\sigma}(\mathbb{R}^{3})\cap L_{\sigma}^{2}(\mathbb{R}^{3})\cap \{\|w_{0}\|_{L^{3}(\mathbb{R}^{3})}<\varepsilon_{0}\}$ for some small positive constant $\varepsilon_{0}$. Then applying the proof of Theorem \ref{DWU08} with a slight modification, we obtain that $w$ is a global $L^{2}$ weak solution.
\end{remark}
\begin{remark}
It is worth remarking that the key to the proof of Theorem \ref{DWU08} lies in applying a bootstrap argument based on the estimates established in Lemmas \ref{Lem01} and \ref{Lem02}.
\end{remark}

\begin{proof}
By Fabes-Jones-Rivi\`{e}re \cite{FJR1972}, we see that $w$ is also a very weak solution in $C([0,T];L_{\sigma}^{p}(\mathbb{R}^{3})).$ Then in order to prove that this solution belongs to $L^{2}$ weak solution, it suffices to demonstrate that $w\in C_{w}([0,T];L^{2}(\mathbb{R}^{3}))\cap L_{T}^{2}\dot{H}_{x}^{1}.$ Let $w=a+z$ as decomposed in \eqref{INT001}. Since $w_{0}\in L_{\sigma}^{p}\cap L_{\sigma}^{2}(\mathbb{R}^{3})$, $p\geq3$, it follows from Lemma \ref{Lem01} that $a\in C_{T}L_{x}^{p}\cap L_{T}^{\frac{4p}{3}}L_{x}^{2p}\cap C_{T}L_{x}^{2}\cap L_{T}^{2}\dot{H}_{x}^{1}$ satisfies
\begin{align}\label{DFM06}
\begin{cases}
\|a\|_{C_{T}L_{x}^{2}}+\|\nabla a\|_{L_{T}^{2}L_{x}^{2}}\leq C(c,\gamma)\|w_{0}\|_{L^{2}},\\
\|a\|_{C_{T}L_{x}^{p}\cap L_{T}^{\frac{4p}{3}}L_{x}^{2p}}\leq C(p,c,\gamma)\|w_{0}\|_{L^{p}}.
\end{cases}
\end{align}
Then in the following we only need to show that $z\in C_{w}([0,T];L^{2}(\mathbb{R}^{3}))\cap L_{T}^{2}\dot{H}_{x}^{1}.$

Multiplying equations \eqref{NONLINE002} with $(w_{1},w_{2})=(w,w)$ by $z$, we have from integration by parts that
\begin{align}\label{FJMQ001}
&\frac{1}{2}\frac{d}{dt}\|z(t)\|_{L^{2}}^{2}+\|\nabla z\|_{L^{2}}^{2}\notag\\
&=\int_{\mathbb{R}^{3}}(z\otimes u^{c,\gamma}+u^{c,\gamma}\otimes z)\cdot\nabla zdx+\int_{\mathbb{R}^{3}}(w\otimes w)\cdot\nabla zdx.
\end{align}
In exactly the same way to \eqref{ZZW003}, we deduce
\begin{align*}
\int_{\mathbb{R}^{3}}(z\otimes u^{c,\gamma}+u^{c,\gamma}\otimes z)\cdot\nabla zdx\leq C_{0}K(c,\gamma)\|\nabla z\|_{L^{2}}^{2}.
\end{align*}
Making use of H\"{o}lder's inequality and Cauchy inequality, we have
\begin{align*}
\int_{\mathbb{R}^{3}}(w\otimes w)\cdot\nabla zdx\leq\|w\otimes w\|_{L^{2}}\|\nabla z\|_{L^{2}}\leq\frac{1}{2}\|\nabla z\|_{L^{2}}^{2}+\frac{1}{2}\|w\|_{L^{4}}^{4}.
\end{align*}
Substituting these two equations into \eqref{FJMQ001} and integrating it from $0$ to $T$, we obtain
\begin{align*}
\|z\|_{C_{T}L_{x}^{2}}+\|\nabla z\|_{L_{T}^{2}L_{x}^{2}}\leq C(c,\gamma)\|w\|_{L_{T}^{4}L_{x}^{4}}^{2},
\end{align*}
where we used the fact that $C_{0}K(c,\gamma)<\frac{1}{2}$. Then it only needs to demonstrate that $w\in L_{T}^{4}L_{x}^{4}$ for the purpose of proving $z\in C_{w}([0,T];L^{2}(\mathbb{R}^{3}))\cap L_{T}^{2}\dot{H}_{x}^{1}.$ According to Theorem \ref{THM001}, we know that $w\in C([0,T];L_{\sigma}^{p}(\mathbb{R}^{3}))\cap L^{\frac{4p}{3}}([0,T];L_{\sigma}^{2p}(\mathbb{R}^{3}))$ satisfies
\begin{align*}
\|w\|_{C_{T}L_{x}^{p}\cap L_{T}^{\frac{4p}{3}}L_{x}^{2p}}+\|\nabla|w|^{\frac{p}{2}}\|^{\frac{2}{p}}_{L^{2}_{T}L^{2}_{x}}\leq C(p,c,\gamma)\|w_{0}\|_{L^{p}(\mathbb{R}^{3})}.
\end{align*}
Then if $p=4$, then we directly have $w\in L_{T}^{4}L_{x}^{4}$. If $p\neq4$, we divide into three cases to discuss as follows.

{\bf Case 1.} When $3\leq p<4$, we have from the interpolation inequality that
\begin{align*}
\|w\|_{L_{T}^{\frac{8p}{3(4-p)}}L_{x}^{4}}\leq\|w\|_{L_{T}^{\infty}L_{x}^{p}}^{\frac{p}{2}-1}\|w\|_{L_{T}^{\frac{4p}{3}}L_{x}^{2p}}^{2-\frac{p}{2}}\leq C\|w_{0}\|_{L^{p}}.
\end{align*}
This, in combination with the fact that $\frac{8p}{3(4-p)}>4$, reads that $w\in L_{T}^{4}L_{x}^{4}$.

{\bf Case 2.} When $4<p\leq8$, using the interpolation inequality again, we obtain
\begin{align}\label{TFE9}
\|w\|_{L_{T}^{\frac{16p}{3(8-p)}}L_{x}^{8}}\leq\|w\|_{L_{T}^{\infty}L_{x}^{p}}^{\frac{p}{4}-1}\|w\|_{L_{T}^{\frac{4p}{3}}L_{x}^{2p}}^{2-\frac{p}{4}}\leq C\|w_{0}\|_{L^{p}}.
\end{align}
In view of $\frac{16p}{3(8-p)}>\frac{16}{3}>\frac{16}{5}$ and utilizing Lemma \ref{Lem02} with $(w_{1},w_{2})=(w,w)$, we obtain from \eqref{TFE9} and H\"{o}lder's inequality that
\begin{align}\label{QDFU09}
\|z\|_{C_{T}L_{x}^{4}\cap L_{T}^{\frac{16}{3}}L_{x}^{8}}\leq& C(c,\gamma)
\|w\|_{L_{T}^{\frac{16}{5}}L_{x}^{8}}\|w\|_{L_{T}^{\frac{16}{3}}L_{x}^{8}}\leq C(c,\gamma,T)\|w_{0}\|_{L^{p}}^{2}.
\end{align}
Since $a\in C_{T}L_{x}^{p}\cap C_{T}L_{x}^{2}$, it then follows from the interpolation inequality that $a\in C_{T}L_{x}^{4}$. This, together with \eqref{QDFU09}, yields that $w\in L_{T}^{4}L_{x}^{4}$.

{\bf Case 3.} Consider the case when $p>8$. Using \eqref{AQ005} and applying Lemma \ref{Lem02} with $p$ replaced by $\frac{p}{2}$, we deduce
\begin{align*}
\|z\|_{C_{T}L_{x}^{\frac{p}{2}}\cap L_{T}^{\frac{2p}{3}}L_{x}^{p}}\leq C(p,c,\gamma)\|w\|_{L_{T}^{\frac{2p}{p-3}}L_{x}^{p}}\|w\|_{L_{T}^{\frac{2p}{3}}L_{x}^{p}}\leq C(p,c,\gamma)\|w_{0}\|_{L^{p}}^{2}.
\end{align*}
With regard to $a$, it follows from \eqref{DFM06} and the interpolation inequality that
\begin{align*}
\|a\|_{C_{T}L_{x}^{\frac{p}{2}}}\leq\|a\|_{C_{T}L_{x}^{p}}^{\frac{p-4}{p-2}}\|a\|_{C_{T}L_{x}^{2}}^{\frac{2}{p-2}}\leq C(p,c,\gamma)\|w_{0}\|_{L^{p}}^{\frac{p-4}{p-2}}\|w_{0}\|_{L^{2}}^{\frac{2}{p-2}},
\end{align*}
and
\begin{align*}
\|a\|_{L_{T}^{\frac{2p}{3}}L_{x}^{p}}\leq&\|a\|_{L_{T}^{\infty}L_{x}^{2}}^{\frac{1}{p-1}}\|a\|_{L_{T}^{\frac{2p(p-2)}{3(p-1)}}L_{x}^{2p}}^{\frac{p-2}{p-1}}\leq T^{\frac{3}{4(p-1)}}\|a\|_{L_{T}^{\infty}L_{x}^{2}}^{\frac{1}{p-1}}\|a\|_{L_{T}^{\frac{4p}{3}}L_{x}^{2p}}^{\frac{p-2}{p-1}}\notag\\
\leq&C(p,c,\gamma)T^{\frac{3}{4(p-1)}}\|w_{0}\|_{L^{2}}^{\frac{1}{p-1}}\|w_{0}\|_{L^{p}}^{\frac{p-2}{p-1}},
\end{align*}
where we also used H\"{o}lder's inequality. Then we have $w\in C_{T}L_{x}^{\frac{p}{2}}\cap L_{T}^{\frac{2p}{3}}L_{x}^{p}.$ Repeating the above bootstrap arguments for finite times, we find an index $i_{0}>0$ such that $w\in C_{T}L_{x}^{\frac{p}{2^{i_{0}}}}\cap L_{T}^{\frac{p2^{2-i_{0}}}{3}}L_{x}^{p2^{1-i_{0}}}$ with $4<\frac{p}{2^{i_{0}}}\leq8$. Then it follows from the arguments in the case of $4<p\leq8$ that $w\in L_{T}^{4}L_{x}^{4}$. The proof is complete.

\end{proof}

\noindent{\bf{\large Acknowledgements.}} X. Zheng was partially supported by the National Natural Science Foundation of China under grant No. 11871087. Z. Zhao was partially supported by CPSF (2021M700358).

%\noindent{\bf{\large Data Availability Statements.}} The data used to support the findings of this study are available from the corresponding author upon request.
%
%\noindent{\bf{\large Conflict of interest.}} The author has no conflict of interest to declare.
%
%\noindent{\bf{\large Acknowledgements.}}
%The author would like to thank Prof. C.X. Miao for his constant encouragement and useful discussions. The author was partially supported
%by CPSF (2021M700358).

\bibliographystyle{plain}

\def\cprime{$'$}

\end{document}